\documentclass{article} % use larger type; default would be 10pt

\makeatletter
\newcommand*\owedge{\mathpalette\@owedge\relax}
\newcommand*\@owedge[1]{%
  \mathbin{%
    \ooalign{%
      $#1\m@th\bigcirc$\cr
      \hidewidth$#1\m@th\wedge$\hidewidth\cr
    }%
  }%
}
\makeatother

\usepackage[utf8]{inputenc} % set input encoding (not needed with XeLaTeX)

\usepackage{graphicx} % support the \includegraphics command and options
\usepackage{booktabs} % for much better looking tables
\usepackage{array} % for better arrays (eg matrices) in maths
\usepackage{paralist} % very flexible & customisable lists (eg. enumerate/itemize, etc.)
\usepackage{verbatim} % adds environment for commenting out blocks of text & for better verbatim
\usepackage{subfig} % make it possible to include more than one captioned figure/table in a single float
\usepackage{mathrsfs}
\usepackage{amssymb}
\usepackage{amsthm}
\usepackage{amsmath,amsfonts,amssymb,esint}
\usepackage{graphics}
\usepackage{enumerate}
\usepackage{mathtools}
\usepackage{xfrac}
\usepackage{bbm}
\usepackage{etoolbox}
\apptocmd{\sloppy}{\hbadness 10000\relax}{}{}
\usepackage{fancyhdr} % This should be set AFTER setting up the page geometry
\usepackage[titletoc,title]{appendix}

\numberwithin{equation}{section}

\newtheorem{Thm}{Theorem}

\newtheorem{Coro}[Thm]{Corollary}

\newtheorem{theorem}{Theorem}[section]
\newtheorem{corollary}[theorem]{Corollary}
\newtheorem{proposition}[theorem]{Proposition}
\newtheorem{lemma}[theorem]{Lemma}
\newtheorem{definition}[theorem]{Definition}

\theoremstyle{definition}

\renewcommand*{\tilde}{\widetilde}

\newcommand{\mb}{\mathbb}
\newcommand{\mc}{\mathcal}

\newcommand{\Rm}{\text{Rm}}

\usepackage[usenames,dvipsnames]{color}
\usepackage[hypertexnames=false, colorlinks=true, pdfstartview=FitV, linkcolor=blue, citecolor=blue, urlcolor=blue]{hyperref}
\usepackage{bookmark}

\title{Convergence of the Ricci flow on asymptotically flat manifolds with integral curvature pinching}
\author{Eric Chen \thanks{Department of Mathematics, Princeton University, Princeton, NJ 08544, USA. Email:
{\footnotesize \href{mailto:ecchen@math.princeton.edu}{ecchen@math.princeton.edu}.}
}
\thanks{Current Address: Department of Mathematics, University of California, Santa Barbara, CA 93106, USA.}}
\begin{document}
\maketitle
\begin{abstract}
    We prove a curvature pinching result for the Ricci flow on asymptotically flat manifolds: if an asymptotically flat manifold of dimension $n\geq 3$ has scale-invariant integral norm of curvature sufficiently pinched relative to the inverse of its Sobolev constant, then the Ricci flow starting from this manifold exists for all positive times and converges to flat Euclidean space. In particular our result implies that the initial manifold must have been diffeomorphic to $\mb{R}^n$.
\end{abstract}

\section{Introduction}

Since the initial work of Hamilton \cite{Ham3}, many curvature pinching results guaranteeing the long-time existence and convergence of the Ricci flow (or its suitable normalization) on compact manifolds have been proved. In this paper a curvature pinching result for the Ricci flow on noncompact, asymptotically flat (AF) manifolds is proved. We first give a short review of known results on the Ricci flow in regard to both curvature pinching and AF manifolds.

For compact manifolds, the long-time existence and convergence of the Ricci flow assuming curvature pinching at the initial time has been well-studied, beginning from Hamilton's proof that the normalized Ricci flow starting from a closed $3$-manifold with $\text{Ric}>0$ exists for all positive times and converges smoothly to a spherical space form \cite{Ham3}. Many other curvature pinching results on the long-time existence and convergence of the Ricci flow on compact manifolds generalizing this statement have since been proved \cite{Huisken,Margerin,Nishikawa,Gursky,HV,BW,BS}, among these for instance the quarter-pinching sphere theorem of Brendle--Schoen.

For noncompact manifolds and in particular AF manifolds, under some assumptions different from curvature pinching long-time existence and convergence results for the Ricci flow have also been proved \cite{OW,SSS}. We remark that AF manifolds are a particularly interesting class of noncompact manifolds to study with regard to the Ricci flow because the flow preserves the AF property as well as the ADM mass \cite{DM}.  There is also a curvature pinching result for the Ricci flow on AF manifolds due to Shi \cite{Shi2}, but B. Chen--Zhu later showed that there exist no AF manifolds satisfying Shi's curvature pinching condition \cite{CZ}.

Our main result, Theorem \ref{A}, is a curvature pinching result for the Ricci flow on AF manifolds for which there exist many nontrivial examples of AF manifolds satisfying its curvature pinching condition. Our other results consist of auxiliary Sobolev inequalities used to prove Theorem \ref{A} as well as statements addressing some questions in the special cases of conformally flat and rotationally symmetric AF manifolds.

The remainder of this introduction has two parts. In the first part we state our results. In the second part we discuss some earlier curvature pinching theorems for the Ricci flow on compact manifolds in order to describe the motivation for our results, and then conclude by outlining the plan of the proofs in the rest of the paper.

\subsection{Results}

We now state the results of this paper. First we state our main theorem on curvature pinching for AF manifolds, Theorem \ref{A}. Then we state the uniform curvature-weighted Sobolev inequalities, Theorems \ref{B} and \ref{B1}, used in the proof of Theorem \ref{A}. Finally we state Corollary \ref{cfc} and Theorem \ref{C1}, which are concerned with the special cases of conformally flat and rotationally symmetric AF manifolds.

\subsubsection{Main theorem}

Below is the main theorem of this paper. We refer to Section \ref{prelim1} for the relevant definitions in the statement below, in particular for the Ricci flow $g(t)$ and the Sobolev constant $C_{g_0}$.

\begin{Thm}\label{A}
Let $(M^n,g_0)$ be an asymptotically flat manifold of order $\tau>0$, with $n\geq 3$. There exists a $\delta(n)>0$ such that if:
\begin{align}
\left(\int_{M^n} |\Rm_{g_0}|^{\frac{n}{2}}\ dV_{g_0}\right)^{\frac{2}{n}}<\delta(n)\frac{1}{C_{g_0}},\label{Ahyp}
\end{align}
where $C_{g_0}$ is the Sobolev constant of $g_0$ given by Definition \ref{SobC}, then the Ricci flow $g(t)$ with initial condition $g(0)=g_0$ exists for all times $t\in[0,\infty)$ and converges in $C^\infty_{-\tau'}(M^n)$ for any $\tau'\in(0,\min(\tau,n-2))$ to the flat $\mb{R}^n$. In particular, $M^n$ is diffeomorphic to $\mb{R}^n$.
\end{Thm}

 This can be viewed as a noncompact analogue of integral curvature pinching results for the Ricci flow on compact manifolds such as \cite{Gursky,HV}. Note that we do not need to assume a priori that $M^n$ is diffeomorphic to $\mb{R}^n$, in contrast to related stability results for the Ricci flow on Euclidean space such as \cite{OW,SSS}.  In Section \ref{exampless} we will see that the integral curvature pinching condition of Theorem \ref{A} is satisfied by many nontrivial examples of AF manifolds, unlike the pointwise curvature pinching condition on AF manifolds of Shi \cite{Shi2}, as discussed earlier. 

We now describe the difficulties we encounter in the two main steps of the proof of Theorem \ref{A}, which are: first, showing  the long-time existence of the Ricci flow; and second, showing the convergence of the flow to flat space once long-time existence is known. 

For the first step of long-time existence we seek to use the weak maximum principle of Moser \cite{Moser} to show that $\|Rm_{g(t)}\|_{L^\infty(M^n)}$ does not blow up along the flow. Therefore, starting from the hypotheses of Theorem \ref{A}, we first need to control integral norms of $\Rm$ along the flow. The main difficulty is that in order to do this we need control of the Sobolev constant $C_{g(t)}$ along the Ricci flow, even though such control cannot be expected in general. To overcome this issue, we study Perelman's $\mc{W}$-functional on AF manifolds and prove a curvature-weighted Sobolev inequality that holds under the assumptions of Theorem \ref{A} along the Ricci flow (Theorem \ref{B1}), inspired by Ye's work on curvature-weighted Sobolev inequalities along the Ricci flow on compact manifolds \cite{RY2}. Despite the curvature weight, we show that this inequality is enough for us to carry out the weak maximum principle argument.

For the second step of convergence of the flow to flat space, the main difficulty is that we need to obtain the much stronger curvature decay estimate $\|Rm_{g(t)}\|_{L^\infty(M^n)}=o(t^{-1})$. To overcome this difficulty we first study the evolution of $\|Rm_{g(t)}\|_{L^q(M^n)}$ for a specific choice of $q=\frac{n}{2}\frac{n}{n-2}$, and we obtain a decay estimate for this quantity using the monotonicity of $\|Rm_{g(t)}\|_{L^{\frac{n}{2}}(M^n)}$, which was shown in the first step. Using this estimate, we can then rework the Moser iteration procedure to find that the specific choice of exponent $q$ gives us exactly the desired $L^\infty$ decay rate of $o(t^{-1})$ for $Rm$. To finish it suffices to adapt work of \cite{YLi} to pass from the decay estimate to convergence to flat space; Li assumes unlike us though that $R_{g(t)}\geq 0$, so we will show how this can be replaced by our assumptions in Theorem \ref{A}.

\subsubsection{Curvature-weighted Sobolev inequalities}

One of the main tools we use in proving Theorem \ref{A} is a curvature-weighted Sobolev inequality which we will also prove, stated below. It is an analogue of a curvature-weighted Sobolev inequality proved for compact manifolds along the Ricci flow by Ye \cite{RY2}, which was obtained using the monotonicity of Perelman's $\mc{W}$-functional \cite{P}. We state two versions of this inequality. In the first version we assume that $R_{g_0}\geq 0$ but require no smallness of $\int_{M^n}|\Rm_{g_0}|^{\frac{n}{2}}\ dV_{g_0}$.

\begin{Thm}\label{B}
Let $(M^n,g(t))$ be a Ricci flow on $[0,T]$ starting from an asymptotically flat manifold $(M^n,g_0)$, $n\geq 3$, of nonnegative scalar curvature. There exists an $A>0$ depending only on $n$ and the Sobolev constant $C_{g_0}$ such that the following weighted Sobolev inequality holds for all $u\in W^{1,2}(M^n,g(t))$ on $[0,T]$:
\begin{align}
    \left(\int_{M^n} |u|^{\frac{2n}{n-2}}\ dV_{g(t)}\right)^{\frac{n-2}{n}}\leq A\left(\int_{M^n}|\nabla_{g(t)} u|^2\ +R_{g(t)} u^2\ dV_{g(t)}\right).\label{WSobolev}
\end{align}
Moreover, in any fixed dimension $n$, $A$ depends linearly on $C_{g_0}$.
\end{Thm}

The second version of the curvature-weighted Sobolev inequality does not require any pointwise assumptions on $R_{g_0}$ but assumes instead that the quantity $\int_{M^n} |\Rm_{g_0}|^{\frac{n}{2}}\ dV_{g_0}$ is small. This is the version we will use to prove Theorem \ref{A}. In fact it suffices to assume just that $\int_{M^n} |R_{g_0}|^{\frac{n}{2}}\ dV_{g_0}$ is small, but for our applications this distinction is not significant. Below, $R^+(x):=\max(0,R(x))$ denotes the positive part of the scalar curvature.

\begin{Thm}\label{B1}
Let $(M^n,g(t))$ be a Ricci flow on $[0,T]$ starting from an asymptotically flat manifold $(M^n,g_0)$, $n\geq 3$. If $\left(\int_{M^n} |\Rm_{g_0}|^{\frac{n}{2}}\ dV_{g_0}\right)^{\frac{2}{n}}<2\frac{1}{C_{g_0}}$, then there exists $A>0$ depending only on $n$ and the Sobolev constant $C_{g_0}$ such that the following weighted Sobolev inequality holds for all $u\in W^{1,2}(M^n,g(t))$ on $[0,T]$:
\begin{align}
    \left(\int_{M^n} |u|^{\frac{2n}{n-2}}\ dV_{g(t)}\right)^{\frac{n-2}{n}}\leq A\left(\int_{M^n}|\nabla_{g(t)} u|^2\ +R_{g(t)}^+ u^2\ dV_{g(t)}\right).\label{Sobolev}
\end{align}
Moreover, in any fixed dimension $n$, $A$ depends linearly on $C_{g_0}$.
\end{Thm}

\subsubsection{Conformally flat AF metrics}

Returning to the statement of Theorem \ref{A}, we have the following consequence for asymptotically flat conformal deformations of the flat metric on $\mb{R}^n$:

\begin{Coro}\label{cfc}
    Let $(\mb{R}^n,g_0)$ be asymptotically flat of order $\tau>0$, with $n\geq 3$ and $g_0=e^{2u}|dx|^2$, where $u\in C^\infty(\mb{R}^n)$. There exists a constant $\Lambda(n)>0$ such that if
    \begin{align}
        \left(\int_{\mb{R}^n}|Rm_{g_0}|^{\frac{n}{2}}\ dV_{g_0}\right)^{\frac{2}{n}}<\Lambda(n),
    \end{align}
    then the Ricci flow $g(t)$ with initial condition $g(0)=g_0$ exists for all times $t\in[0,\infty)$ and converges in $C^\infty_{-\tau'}(\mb{R}^n)$ for any $\tau'\in(0,\min(\tau,n-2))$ to the flat metric on $\mb{R}^n$.
\end{Coro}

In particular the above statement holds for rotationally symmetric asymptotically flat metrics on $\mb{R}^n$, since these are conformal deformations of the flat metric. Such metrics are exactly those studied by \cite{OW}, and we will briefly discuss the following connection between our scale-invariant $L^{\frac{n}{2}}$ curvature pinching condition and the no minimal hyperspheres condition of \cite{OW} guaranteeing long-time existence and convergence of the flow in the rotationally symmetric asymptotically flat setting, which we have quoted in Theorem \ref{OWt}. It is easily obtained in even dimensions $n$ as a consequence of the Chern-Gauss-Bonnet formula for manifolds with boundary \cite{AW}, but we will see that it holds in any dimension $n\geq 2$.

\begin{Thm}\label{C1}
For every $n\geq 2$, there exists a $C(n)>0$ such that if $(\mb{R}^n,g)$ is a rotationally symmetric metric with $\int_{\mb{R}^n}|\Rm_g|^{\frac{n}{2}}\ dV_g<C(n)$, then $(\mb{R}^n,g)$ does not contain any minimal hyperspheres.
\end{Thm}

Whether the special case of Corollary \ref{cfc} in the rotationally symmetric setting is a consequence of the work of \cite{OW} or not would depend on a comparison of the sizes of the constants $\Lambda(n)$ and $C(n)$, but we will not further pursue this question in this paper.

\subsection{Some previous results and motivation}

From among the variety of curvature pinching results known on compact manifolds, below we will look specifically at the integral curvature pinching results of \cite{Gursky,HV} in more detail in order to describe how those results help motivate the study in this paper. This will also lead us to examples of nontrivial AF manifolds which satisfy the curvature pinching condition in Theorem \ref{A}.

\subsubsection{Integral curvature pinching on compact manifolds}

On compact manifolds, the long-time existence and convergence of Ricci flow was studied first under pointwise curvature pinching hypotheses, such as in \cite{Ham3,Huisken,Margerin,Nishikawa}, which were later partially generalized to integral curvature pinching hypotheses, including in \cite{Ye1,Gursky,HV,DPW}. In what follows we want to discuss in particular some statements of \cite{Gursky,HV};  $\mc{Y}^+(M^n)$ will denote the set of positive scalar curvature Yamabe metrics of a closed manifold $M^n$ with unit volume. A special case of a result of Gursky's showed the following:
\begin{theorem}[{\cite{Gursky}}]\label{Gky}
        There exists $\epsilon>0$ such that if $g_0\in\mc{Y}^+(S^4)$ satisfies $\int_{S^4}|W_{g_0}|^2\ dV_{g_0}<\epsilon$, then the normalized Ricci flow $g(t)$ with initial condition $g(0)=g_0$ exists for all times $t\in[0,\infty)$ and converges to the standard $S^4$.
\end{theorem}

Later, Hebey--Vaugon showed the following result in the same spirit, but with less restriction on the underlying manifold:

\begin{theorem}[{\cite{HV}}]\label{HVthm}
        For $n\geq 3$ and $q>\frac{n}{2}$, there exists $\epsilon=\epsilon(n,q)>0$ such that if $(M^n,g_0)$ is a closed manifold with $g_0\in\mc{Y}^+(M^n)$ which satisfies $\left(\int_{M^n}\left|W_{g_0}+\frac{1}{n-2}E_{g_0}\owedge g_0\right|^q\ dV_{g_0}\right)^{\frac{1}{q}}<\epsilon R_{g_0}$, then the normalized Ricci flow $g(t)$ with initial condition $g(0)=g_0$ exists for all times $t\in[0,\infty)$ and converges to a spherical space form.
\end{theorem}

Above, $W$ denotes the Weyl tensor and $E$ denotes the traceless Ricci tensor. We will now point out how the integral curvature pinching hypothesis in Theorem \ref{A} can be seen as a natural analogue of the hypotheses of Theorem \ref{Gky} and \ref{HVthm}. First note that by the Chern--Gauss--Bonnet theorem, the hypothesis  $\int_{S^4}|W_{g_0}|^2\ dV_{g_0}<\epsilon$ from Gursky's theorem may be rewritten in a form analogous to that of the hypothesis in the result of Hebey--Vaugon:
\begin{align}
\left(\int_{S^4}\left|W_{g_0}+\frac{1}{2}E_{g_0}\owedge g_0\right|^2\ dV_{g_0}\right)^{\frac{1}{2}}<\tilde{\epsilon} R_{g_0},\quad\text{for some }\tilde{\epsilon}>0.
\end{align}
Furthermore, note that by the assumption $g_0\in\mc{Y}^+(M^n)$ we have $R_{g_0}=Y(M^n,[g_0])$, which is the inverse of its conformally invariant Sobolev constant in the sense that for all $u\in W^{1,2}(M^n,g_0)$,
\begin{align}
    \left(\int_{M^n} |u|^{\frac{2n}{n-2}}\ dV_{g_0}\right)^{\frac{n-2}{n}}\leq \frac{1}{Y(M^n,[g_0])}\left(\int_{M^n}4\frac{n-1}{n-2}|\nabla u|^2+R_{g_0}u^2\ dV_{g_0}\right).
\end{align}
Recall now the curvature decomposition $Rm=W+\frac{1}{n-2} E\owedge g+\frac{R}{2 m(m-1)} g\owedge g$. On spherical space forms, the first two components of this decomposition vanish, while the third does not (the scalar curvature is a positive constant). Thus, the results of Gursky and Hebey--Vaugon state that if a compact manifold whose components of curvature which vanish on spherical space forms ($W$ and $E$) are small in an integral sense relative the inverse of its conformally invariant Sobolev constant, then the normalized Ricci flow starting from that manifold exists for all positive times and converges to a spherical space form. For flat Euclidean space, all three components of the curvature decomposition vanish. Thus Theorem \ref{A} can be viewed as stating in the spirit of \cite{Gursky,HV} (and also \cite[Theorem 1.2]{DPW}) that if an AF manifold whose components of curvature which vanish on flat Euclidean space (all of them) are small in an integral sense relative the inverse of its Sobolev constant, then the Ricci flow starting from that manifold exists for all positive times and converges to flat Euclidean space.

\subsubsection{Stereographic projection and integral curvature pinching for AF manifolds}\label{exampless}

We will now see another way in which the curvature pinching results in the compact setting quoted above, in particular the result of Gursky, Theorem \ref{Gky}, lead naturally to the question of whether an integral curvature pinching hypothesis as in our main result, Theorem \ref{A}, can guarantee long-time existence and convergence of the Ricci flow on an AF manifold to flat space. This will also result in many nonflat examples of AF manifolds satisfying the hypotheses of Theorem \ref{A}.

Consider metrics $g_0\in\mc{Y}^+(S^4)$ with $\int_{S^4}|W_{g_0}|^2\ dV_{g_0}<\epsilon$ satisfying the hypotheses of Gursky's theorem. Then, using the Green's function of the conformal Laplacian with respect to any $p\in S^4$ gives a conformal change of the initial metric $g_0$ on $S^4\backslash\{p\}$ to an AF metric $\tilde{g_0}$ viewed on $\mb{R}^4$ with vanishing scalar curvature. Therefore, by applying the Chern--Gauss--Bonnet formula on $S^4$ and the conformal invariance of the Yamabe constant, it can be seen that this AF metric's Sobolev constant (see Definition \ref{SobC} for notation) is bounded from above by an absolute constant if we assume $\epsilon<1$. Moreover, $\int_{\mb{R}^4}|\Rm_{\tilde{g_0}}|^2\ dV_{\tilde{g_0}}$ is small.
To see this, apply the Chern--Gauss--Bonnet formula for asymptotically locally Euclidean spaces \cite{TV} to the AF setting:
\begin{align}
\frac{1}{4}\int_{\mb{R}^4} |W_{\tilde{g_0}}|^2\ dV_{\tilde{g_0}}-\frac{1}{2}\int_{\mb{R}^4}|E_{\tilde{g_0}}|^2\ dV_{\tilde{g_0}}+\frac{1}{24}\int_{\mb{R}^4} R^2_{\tilde{g_0}}\ dV_{\tilde{g_0}}=0.
\end{align}
Since $R_{\tilde{g_0}}\equiv 0$, and $\int|W|^2\ dV$ is conformally invariant in dimension four, we see that $\int_{\mb{R}^4}|E_{\tilde{g_0}}|^2\ dV_{\tilde{g_0}}<\frac{1}{2}\epsilon$ and so indeed $\int_{\mb{R}^4}|\Rm_{\tilde{g_0}}|^2\ dV_{\tilde{g_0}}$ is small. 

Gursky's result tells us that on $S^4$, the metric $g_0$ evolves under the Ricci flow to the round metric on the sphere, which stereographically projects to the flat $\mb{R}^4$. We might consider instead whether the noncompact, generalized stereographic projection of $(S^4\backslash\{p\},g_0)$ would evolve under the Ricci flow to flat space; note that for $p$ fixed, our stereographically projected metric depends only on the conformal class $[g_0]$ on $S^4$ (up to constant scaling), which is different from the situation on $S^4$, where Gursky's result picks out the Yamabe metric out of the many metrics belonging to $[g_0]$ as a starting point for the Ricci flow. From these considerations we are led to the following more general question for AF manifolds: Does the Ricci flow starting from an AF manifold $(M^n,g_0)$ with bounded Sobolev constant and $\int|\Rm_{g_0}|^{\frac{n}{2}}\ dV_{g_0}$ sufficiently small converge to flat Euclidean space?  This question is affirmatively answered by our main result, Theorem \ref{A}. 

\subsubsection{Plan of the paper}

In the rest of this paper we will present the proofs for our results stated in this second part of the introduction. In Section \ref{prelim1} we will state some preliminary notions and the known short-time existence and uniqueness results for the Ricci flow on asymptotically flat manifolds. In Section \ref{Bsec} we will prove Theorem \ref{B} using Perelman's $\mc{W}$-functional along with heat semigroup arguments originally adapted to compact manifolds by Ye \cite{RY2}; modifying this proof we will also obtain Theorem \ref{B1}. In Section \ref{Asec} we will prove Theorem \ref{A} in dimensions $n\geq 4$; using Theorem \ref{B1} we will obtain $L^{\frac{n}{2}}$ and higher $L^q$ boundedness of $|\Rm|$ for a particular choice of $q>\frac{n}{2}$, and then we will prove the long-time existence of the Ricci flow $g(t)$ on $[0,\infty)$ by showing that $|\Rm|$ is uniformly bounded along the flow using Moser iteration. Then, using the long-time existence of the flow we conclude the proof of Theorem \ref{A} when $n\geq 4$ by proving the much stronger decay estimate $\lim_{t\rightarrow\infty}\sup_{x\in M^n}t|\Rm(t)|=0$; at this point a straightforward adaptation of work of Li then implies that $(M^n,g(t))$ converges to $(\mb{R}^n,g(t))$ \cite{YLi}. In Section \ref{n3} we address a technical point regarding the proof of Theorem \ref{A} in dimension $n=3$ and show how the arguments of the previous section can be modified to complete the proof in this remaining case. We end with an Appendix in which we show how Corollary \ref{cfc} follows from Theorem \ref{A} and further discuss it in the context of rotationally symmetric metrics by proving Theorem \ref{C1}.

Regarding notation, $M^n$ and $M$ will refer to the same manifold; the superscript $n$ just emphasizes the dimension of $n$. We will often use $C_0$, $R_0$ instead of $C_{g_0}, R_{g_0}$ and so on to denote quantities associated with the metric $g_0=g(0)$ along a Ricci flow, as well as $C_t$, $R_t$ for quantities associated with $g(t)$ in a similar way. Norms such as $|\text{Rm}|$ along the Ricci flow are computed with respect to the same metric at time $t$ from which $\text{Rm}$ arises, unless otherwise indicated, and we adopt a similar convention when denoting function spaces such as $W^{1,2}(M)$. We will always be integrating over the entirety of the manifold under discussion, so we will no longer use subscripts to specify this unless necessary. Constants $C_1$, $C_2$, and so on which appear below, primarily in Section \ref{Asec}, are strictly positive and may change in size between lines but will not depend on quantities involved in the assumptions of Theorem \ref{A} with the exception of the dimension $n$ (unless otherwise indicated). Dependence on other constants that may appear in computations will be denoted by $C_1(\alpha)$ for dependence on $\alpha$, for example.

\section*{Acknowledgments}
I wish to thank my advisor, Sun-Yung Alice Chang, for many helpful discussions and comments as well as her constant support and encouragement. I also wish to thank Paul Yang for his support and many helpful discussions, as well as Eric Woolgar for several helpful correspondences. This paper is a minor revision of work first presented in the author's doctoral thesis at Princeton University, supported by the National Science Foundation Graduate Research Fellowship under Grant No. DGE-1656466.

\section{Preliminaries}\label{prelim1}

\subsection{Definitions}

We now introduce some notations and definitions that we will need later. First we state the definition we will use for an asymptotically flat manifold.

\begin{definition}\label{AFdef}
We call a smooth Riemannian manifold $(M^n,g)$ asymptotically flat of order $\tau>0$ if for some compact set $K\subset M^n$, there exists an $R>0$ and a diffeomorphism $\Phi:M^n\backslash K\rightarrow \mb{R}^n\backslash B_R(0)$ such that for some $\tau>0$,
\begin{align}
g_{ij}(x)=\delta_{ij}+O(|x|^{-\tau})\quad\text{and}\quad \partial^\alpha g_{ij}(x)=O(|x|^{-(\tau+|\alpha|)}),\label{decaycond}
\end{align}
for partial derivatives $\partial^\alpha$ of any order, as $|x|\rightarrow\infty$ in $\mb{R}^n$. We call $\Phi$ the asymptotically flat coordinate system.
\end{definition}

We will not mention the order of an asymptotically flat manifold unless it is specifically needed. Note however that we always have $\tau>0$ in our discussion. Also note that by \eqref{decaycond} we have for any asymptotically flat $(M,g)$ of order $\tau>0$ that $\int|\Rm|^{p}\ dV_g<\infty$ for all $p\geq\frac{n}{2}$, and in fact for all $p>\frac{n}{2+\tau}$. Next we define the notion of convergence obtained in Theorem \ref{A}.

\begin{definition}\label{Cspace}
Let $(M^n,g)$ be an asymptotically flat manifold of order $\tau$. For any $\beta\in\mb{R}$ and nonnegative integer $k$, the space $C_\beta^k(M)$ is given by the $C^k$ functions on $M$ for which the norm
\begin{align}
\|u\|_{C_\beta^k}=\sum_{i=0}^k\sup_M r^{-\beta+i}|\nabla^i u|<\infty,
\end{align}
where $r\in C^\infty(M)$ is a smooth positive function with $r=|x|$ on $M^n\backslash K$, in the notation of Definition \ref{AFdef}. We say that we have convergence in $C^\infty_\beta$ if we have convergence in $C^k_\beta$ for all $k\geq 0$.
\end{definition}

Now we define our convention of notation for the Sobolev constant $C_g$ associated to an asymptotically flat manifold $(M^n,g)$.

\begin{definition}\label{SobC}
If $(M^n,g)$ is an asymptotically flat manifold, then there exists a smallest constant $C_g>0$ such that for every $u\in W^{1,2}(M,g)$, the following Sobolev inequality holds:
\begin{align}
\left(\int |u|^{\frac{2n}{n-2}}\ dV_g\right)^{\frac{n-2}{n}}\leq C_g\int|\nabla u|^2\ dV_g.\label{Sob1}
\end{align}
We call $C_g$ the Sobolev constant of the metric $g$.
\end{definition}

The constant $C_g>0$ above exists because any asymptotically flat manifold satisfies an isoperimetric inequality, and this implies the validity of the Euclidean Sobolev-type inequality \eqref{Sob1}. We may also see that \eqref{Sob1} holds in the following way: \cite[Proposition 2.5]{Carron} implies that \eqref{Sob1} is valid if it holds for all $u\in C_0^\infty(M\backslash K)$, where $K\subset M$ is some compact set, and the validity of this second inequality can be easily seen using the asymptotically flat coordinate system, since $(\Phi^{-1})^*g$ is uniformly equivalent to $|dx|^2$ on the set $\Phi^{-1}(\mb{R}^n\backslash B_R(0))$, for $R>0$ sufficiently large. It is well known that $C_{n,e}\leq C_g$, where $C_{n,e}$ is the Sobolev constant for the flat metric on $\mb{R}^n$ (see for instance \cite[Proposition 4.2]{Hebeyb}).

\subsection{Short-time existence of the Ricci flow for AF manifolds}

The Ricci flow evolves metrics $g(t)$ on a manifold $M^n$ by the equation $\partial_t g(t)=-2 Ric_{g(t)}$. Short time existence, uniqueness, and blowup alternative results for the Ricci flow on compact manifolds were proved by Hamilton \cite{Ham3, HamS}. In the complete noncompact setting, Shi proved the short-time existence of the Ricci flow for initial metrics of bounded curvature \cite{Shi}, and later Chen--Zhu show the uniqueness of Ricci flows of bounded curvature with the same initial data \cite{CZ}. These results have been adapted to the asymptotically flat setting, where the asymptotically flat condition has been shown to be preserved along the Ricci flow \cite{OW,YLi}. This is summarized in the following short-time existence and uniqueness statement, which describes exactly those Ricci flows which we consider in this paper.

\begin{theorem}[{\cite{OW,YLi}}]\label{shorttime}
Let $(M^n,g_0)$ be an asymptotically flat manifold of order $\tau>0$. There exists a unique Ricci flow $g(t)$ with initial condition $g(0)=g_0$ on a maximal time interval $0\leq t<T_M\leq\infty$ such that $g(t)$ remains asymptotically flat of the same order $\tau>0$ with the same asymptotically flat coordinate system, and if $T_M<\infty$ then
\begin{align}
\limsup_{t\rightarrow T_M}\sup_{x\in M}|\text{Rm}|&=\infty.\label{blowup}
\intertext{In particular, for any $T\in[0,T_M)$, it holds that}
\sup_{\substack{t\in[0,T]\\x\in M}}|\text{Rm}|&<\infty,
\end{align}
and the metrics $g(t)$ are equivalent for all $t\in[0,T]$. More precisely
\begin{align}
e^{-2 KT}g_0\leq g(t)\leq e^{2KT}g_0,
\end{align}
where $K=\sup_{\substack{t\in[0,T]\\x\in M}}|\text{Rm}(t,x)|_{g(t,x)}<\infty$.
\end{theorem}

To conclude this section we point out some uniform estimates that the Ricci flow of Proposition \ref{shorttime} satisfies on closed intervals $[0,T]$ on which it exists.

\begin{theorem}[{\cite{Shi,YLi}}]\label{shortbound}
    Let $(M^n,g_0)$ be an asymptotically flat manifold of order $\tau>0$. Suppose that the Ricci flow as in Proposition \ref{shorttime} exists on the interval $[0,T]$. Then for all $k=0,1,2,\ldots$, there exist constants $C_k>0$ depending on $T$ such that on $[0,T]\times M$,
    \begin{align}
        |\nabla^k\Rm(t,x)|\leq C_k r^{-2-k-\tau}.
    \end{align}
\end{theorem}
These estimates will be important in allowing us to consider the evolution of curvature integrals over $(M^n,g(t))$ when the flow exists, primarily in the arguments of Section \ref{Asec}.

\section{Uniform weighted Sobolev inequalities along the flow}\label{Bsec}

In this section we will prove the curvature-weighted Sobolev inequality of Theorem \ref{B}, and then conclude by showing how to adapt the proof to prove the related inequality of Theorem \ref{B1}. The general idea comes from the method of Ye \cite{RY2}, who obtains some uniform curvature-weighted Sobolev inequalities in the compact setting along the Ricci flow. In his work, he first derives a log Sobolev inequality at $t=0$, applies the well-known monotonicity of Perelman's $\mc{W}$-functional in the compact case to obtain a curvature-weighted log Sobolev inequality along the Ricci flow, and then adapts heat semigroup arguments of Davies \cite{Davies} to the setting of compact manifolds to conclude a uniform curvature-weighted Sobolev inequality along the flow (under suitable assumptions on the initial compact manifold).

We will follow the same outline, under the hypotheses of Theorems \ref{B} or \ref{B1}. First, simple estimates give a log Sobolev inequality for $g(0)$ as a consequence of the Sobolev inequality for $g(0)$. Second, we translate this log Sobolev inequality at $t=0$ into a log Sobolev inequality at later times $t$ using the monotonicity of Perelman's $\mc{W}$-functional, but we need to justify the monotonicity of the $\mc{W}$-functional in our particular noncompact asymptotically flat setting. Moreover, the Sobolev and log Sobolev inequalities we work with in the noncompact case are different from those that Ye considers in the compact case. To conclude \eqref{WSobolev} we apply semigroup arguments just as \cite{Davies,RY2} but on an exhaustion of bounded sets $B_N(0)$ covering $M^n$ as $N\rightarrow\infty$ to pass from the log Sobolev to the Sobolev inequality, referring to those works for details.

\subsection{Log Sobolev inequalities and the entropy functional}

\subsubsection{Log Sobolev inequality at \texorpdfstring{$t=0$}{t=0}}

The first fact we will need is that a log Sobolev inequality holds for any asymptotically flat manifold $(M^n,g)$:

\begin{lemma}\label{logSob}
Let $(M^n,g)$ be an asymptotically flat manifold with Sobolev constant $C_g$ (recall Definition \ref{SobC}). Then the following log Sobolev inequality holds for all $u\in W^{1,2}(M,g)$ satisfying $\int u^2\ dV_g=1$:
\begin{align}
\int u^2\log u^2\ dV_g\leq 4\tau\int|\nabla u|^2\ dV_g-\frac{n}{2}\log\tau+\frac{n}{2}\left(\log C_g+\log\frac{n}{8}-1\right),\label{logSob1}
\end{align}
\end{lemma}
\begin{proof}
We apply Jensen's inequality to \eqref{Sob1}, obtaining
\begin{align}
\frac{n-2}{n}\int u^2\log |u|^{\frac{4}{n-2}}\ dV_g\leq\log C_g+\log\int|\nabla u|^2\ dV_g.\notag
\end{align}
The conclusion follows after we apply to the second term on the right the elementary inequality (see \cite[Lemma 3.2]{RY2} for a short proof),
\begin{align}
\log A\leq 2\sigma A-1-\log2\sigma,\quad\text{for any }A>0,\ \sigma>0.\notag
\end{align}
\end{proof}

\subsubsection{Perelman's \texorpdfstring{$\mc{W}$}{W}-functional and monotonicity along the flow}

Next we will establish the monotonicity of Perelman's $\mc{W}$-functional in our asymptotically flat setting. Recall that the $\mc{W}$-functional \cite{P} is the following quantity on a Riemannian manifold $(M,g)$:
\begin{align}
\int \left[\tau(|\nabla f|^2+R)+f-n\right]\frac{e^{-f}}{(4\pi\tau)^{\frac{n}{2}}}\ dV_g,
\end{align}
given $\tau>0$ and $f\in C^\infty(M^n)$ satisfying
\begin{align}
\int \frac{e^{-f}}{(4\pi\tau)^{\frac{n}{2}}}\ dV_g=1.\notag
\end{align}
Making the change of variable $u=\frac{e^{-\frac{f}{2}}}{(4\pi\tau)^{\frac{n}{4}}}$, we write
\begin{align}
\mc{W}(g,u,\tau)=\int \left[\tau(4|\nabla u|^2+Ru^2)-u^2\log u^2\right]\ dV_g -n-\frac{n}{2}\log 4\pi\tau,
\end{align}
where $\int u^2 dV_g=1$. We also define
\begin{align}
\mu(g,\tau)=\inf_{\substack{u\in W^{1,2}\\\int u^2\ dV_g=1}}\mc{W}(g,u,\tau).
\end{align}
Note for asymptotically flat manifolds that $\mu(g,\tau)$ is bounded below for any fixed $\tau>0$ by \eqref{logSob1}, since we have Lemma \ref{logSob}, and we can bound the additional term involving scalar curvature in the $\mc{W}$-functional from below by $\int\tau R u^2\ dV_g\geq\tau\inf_M R$. Therefore $\mu(g,\tau)$ is well defined in our setting. Furthermore in the definition of $\mu(g,\tau)$ it suffices to consider $u\in C_0^\infty(M)$, which is dense in $W^{1,2}(M)$ due to a result of Aubin (see for instance \cite[Theorem 3.1]{Hebey}). Alternatively we may also consider positive functions $u\in C^\infty(M)$ with quadratic exponential decay whose derivatives also have quadratic exponential decay,
\begin{align}
&|u(x)|\leq C_0 e^{-B_0 d_g(0,x)^2},\notag
\\
&|\nabla^p u(x)|\leq C_p e^{-B_p d_g(0,x)^2},\quad p=1,2,\ldots.\notag
\end{align}
We denote such functions below by $\mc{C}_+(M)$.

\begin{lemma}\label{quadid}
Let $g(t)$ be a solution of the Ricci flow defined on $[0,T]$ starting from an asymptotically flat manifold $(M^n,g_0)$. Let $L>T$, and consider $\tau(t)$, $v(t,x)$ satisfying
\begin{align}
\begin{cases}
\tau(t)=L-t,\quad t\in[0,T],
\\
\frac{\partial v}{\partial t}=-\Delta v+Rv,\quad(t,x)\in [0,T]\times M,
\end{cases}
\end{align}
where $v(T)>0$ is a given terminal value in $\mc{C}_+(M)$ satisfying $\int v(T)\ dV_t=1$. Let $u(t,x)=\sqrt{v(t,x)}$. Then Perelman's monotonicity formula holds for $\mc{W}=\mc{W}(g(t),u(t),\tau(t))$ on $[0,T]$,
\begin{align}
\frac{d\mc{W}}{dt}=2\tau\int\left|\text{Ric}-\nabla^2 \log v-\frac{1}{2\tau} g\right|^2v\ dV_t\geq 0.\label{monoform}
\end{align}
\end{lemma}
\begin{proof}
The proof proceeds in the same way as the proof of \cite[Corollary 4.1]{QZ}. There, \eqref{monoform} was proved when the terminal value $v(T)$ was a specific function whose existence was shown in an earlier part of the paper; however, the comparison arguments and heat kernel estimates used in \cite{QZ} to prove the formula only used the fact that $v(T)$ and $|\nabla v(T)|$ have quadratic exponential decay with respect to $|x|$, along with an assumption of bounded geometry --- that is, $\Rm$ and its derivatives bounded as well as a lower bound for metric balls of radius one.

In our case, $v(T)$ and $|\nabla v(T)|$ do indeed have quadratic exponential decay with respect to $|x|$ by assumption. For bounded curvature, observe that by the short-time existence results of Proposition \ref{shorttime}, we have that $|\Rm|$ is uniformly bounded on $[0,T]$ and the metrics $g(t)$ are also equivalent on this interval, hence
\begin{align}
    &\sup_{\substack{t\in[0,T]\\x\in M}}| \text{Rm}|\leq C,\notag
    \\
    &\inf_{\substack{t\in[0,T]\\x\in M}}\text{Vol}(B_x(1))\geq\beta>0,\notag
\end{align}
for some constants $C$, $\beta>0$. The boundedness of all higher derivatives of $\Rm$ then follows by applying a noncompact maximum principle to the parabolic inequality satisfied by $|\Rm|^2$:
\begin{align}
    \partial_t|\Rm|^2\leq\Delta|\Rm|^2+16|\Rm|^3.
\end{align}
Details of such an argument may be found for instance in \cite{YLi}. Therefore the comparison arguments and heat kernel estimates used in \cite{QZ} apply in our setting as well to conclude the proof.
\end{proof}

As a consequence we obtain monotonicity of the quantity $\mu(g,\tau)$ along the Ricci flow:

\begin{lemma}\label{mono}
Let $g(t)$ be a solution of the Ricci flow defined on $[0,T]$ starting from an asymptotically flat manifold $(M^n,g_0)$. Let $L>T$, and set $\tau=L-t$. Then $\mu(g(t),\tau(t))$ is nondecreasing in $t$.
\end{lemma}
\begin{proof}
We have earlier remarked that in the definition of $\mu$ as an infimum it suffices to consider functions $u\in\mc{C}_+(M)$. Therefore, if there existed $0\leq t_1<t_2\leq T$ such that $\mu(g(t_1),\tau(t_1))>\mu(g(t_2),\tau(t_2))$, then applying the monotonicity formula of Lemma \ref{quadid} by specifying terminal values $v_k(t_2)\in\mc{C}_+(M)$, where $\{v_k\}$ are such that  $\mc{W}(g(t_2),\sqrt{v_k(t_2)},\tau(t_2))\xrightarrow{k\rightarrow\infty}\mu(g(t_2),\tau(t_2))$, would give a contradiction when $k$ is sufficiently large.
\end{proof}

\subsection{Log Sobolev inequalities along the flow}\label{Sobsec}
\subsubsection{Proof of Theorem \ref{B}}

Applying the above results, we are now able to prove Theorem \ref{B}.

\begin{proof}[Proof of Theorem \ref{B}]
By Lemma \ref{logSob} we have the log Sobolev inequality \eqref{logSob1} for the metric $g_0$ at time $t=0$, for all $u\in W^{1,2}(M)$ with $\int u^2\ dV_0=1$. Using the monotonicity of $\mc{W}$ we will now translate this into a Sobolev inequality at time $t\in[0,T]$. We consider the functional $\mc{W}^*$ as in \cite{RY2},
\begin{align}
    \mc{W}^*(g,u,\tau)&=\int \tau(4|\nabla u|^2+R u^2)-u^2\log u^2\ dV_g
    \\
    &=\mc{W}(g,u,\tau)+2\log\tau+2\log 4\pi+4,\notag
\end{align}
and the associated quantity
\begin{align}
\mu^*(g,\tau)=\inf_{\substack{u\in W^{1,2}\\ \int u^2\ dV_g=1}}\mc{W^*}(g,u,\tau).
\end{align}
Notice that $\mc{W}^*$ is equal to $\mc{W}$ up to an additive constant plus a term depending on $\tau$. Then, arguing as in the proof of Lemma \ref{mono} and accounting for this extra term, we find that for any $\sigma>0$ and $t\in[0,T]$,
\begin{align}
    \mu^*(g(t),\sigma)\geq \mu^*(g(0),t+\sigma)+\frac{n}{2}\log\frac{\sigma}{t+\sigma}.\label{mumon}
\end{align}
Now, since $R\geq 0$ at time $t=0$, we can add in a scalar curvature-weighted term in \eqref{logSob1} to see that
\begin{align}
\mu^*(g(0),t+\sigma)\geq \frac{n}{2}\log(t+\sigma)-\frac{n}{2}\left(\log C_0+\log\frac{n}{8}-1\right),\label{munext}
\end{align}
while by definition,
\begin{align}
\mu^*(g(t),\sigma)\leq \int\sigma(4|\nabla u|^2+Ru^2)-u^2\log u^2\ dV_t,
\end{align}
for any $u\in W^{1,2}(M)$ satisfying $\int u^2\ dV_t=1$. The following log Sobolev inequality therefore holds at time $t$ for any $t\in[0,T]$ and $\sigma>0$:
\begin{align}
    \int u^2\ln u^2\ dV_t\leq \sigma\int|\nabla u|^2+\frac{R}{4} u^2\ dV_t-\frac{n}{2}\log\sigma+\frac{n}{2}\left(\log C_0+\log\frac{n}{8}-1\right),\label{logall}
\end{align}
for any $u\in W^{1,2}(M)$ with $\int u^2\ dV_t=1$. In particular, such an inequality holds on $W_0^{1,2}(B_N(0))\subset W^{1,2}(M)$ if $\int u^2\ dV_t=1$ for any $N>0$. The arguments of \cite[Appendix 3]{RY2}, then adapt to show that there exists an $A>0$ depending only on $C_0$ and $n$ such that for any $u\in W_0^{1,2}(B_N(0))$ we have the Sobolev inequality
\begin{align}
\left(\int |u|^{\frac{2n}{n-2}}\ dV_t\right)^{\frac{n-2}{n}}\leq A\left(\int|\nabla u|^2\ +\frac{R}{4}u^2\ dV_t\right).
\end{align}
In particular tracing the constant dependence in \cite[Theorems 5.4--5.5]{RY2} we find that $A$ depends linearly on $C_0$ for fixed $n$. Taking $N\rightarrow\infty$, we may therefore obtain \eqref{WSobolev}.
\end{proof}

\subsubsection{Proof of Theorem \ref{B1}}

We conclude by showing how the proof of Theorem \ref{B} can be modified to prove Theorem \ref{B1}.

\begin{proof}[Proof of Theorem \ref{B1}]
Proceeding just as in the proof of Theorem \ref{B}, we again arrive at the inequality \eqref{mumon} without having used the assumption $R_0\geq 0$. In the next step \eqref{munext} where $R_0\geq 0$ was used, we instead consider
\begin{align}
&4\tau\int|\nabla u|^2\ dV_0+2\tau\int Ru^2\ dV_0-\frac{n}{2}\log\tau+\frac{n}{2}\left(\log C_0+\log\frac{n}{8}-1\right)\label{noR}
\\
&\qquad \geq \int u^2\log u^2\ dV_0+2\tau\int Ru^2\ dV_0.\notag
\end{align}
Then we observe that
\begin{align}
2\tau\int Ru^2\ dV_0&\geq -2\tau\left(\int |R|^{\frac{n}{2}}\ dV_0\right)^{\frac{2}{n}}\left(\int |u|^{\frac{2n}{n-2}}\ dV_0\right)^{\frac{n-2}{n}}
\\
&\geq -2\tau C_0\left(\int |R|^{\frac{n}{2}}\ dV_0\right)^{\frac{2}{n}}\int|\nabla u|^2\ dV_0.\notag
\end{align}
Hence if $\left(\int R^{\frac{n}{2}}\ dV_0\right)^{\frac{2}{n}}<2\frac{1}{C_0}$ then we can absorb it on the left of \eqref{noR} to obtain
\begin{align}
&8\tau\int|\nabla u|^2\ dV_0+2\tau\int Ru^2\ dV_0-\frac{n}{2}\log2 \tau+\frac{n}{2}\left(\log C_0+\log\frac{n}{8}-1\right)\label{mulb}
\\
&\qquad \geq \int u^2\log u^2\ dV_0-\frac{n}{2}\log 2.\notag
\end{align}
As a result we may continue as in the proof of Theorem \ref{B} to obtain \eqref{logall} with a different constant in the last term:
\begin{align}
\int u^2\ln u^2\ dV_t\leq \sigma\int|\nabla u|^2+\frac{R}{4} u^2\ dV_t-\frac{n}{2}\log\sigma+\frac{n}{2}\left(\log C_0+\log\frac{n}{4}-1\right).
\end{align}
We conclude by proceeding through the arguments of \cite[Appendix 3]{RY2} as adapted to our case by studying the semigroup associated with the positive operator $-\Delta+R^+$ on $B_N(0)$, instead of $-\Delta+R$ in the case when $R\geq 0$, and taking $N\rightarrow\infty$. As at the conclusion of the proof of Theorem \ref{B}, we find that $A$ depends linearly on $C_0$ for fixed $n$.
\end{proof}

\section{Long-time existence and convergence of the flow, \texorpdfstring{$n\geq 4$}{n=4}}\label{Asec}

We will now prove Theorem \ref{A} in dimensions $n\geq 4$; as mentioned in the introduction, in dimension $n=3$ there are some additional technical considerations that we will address later in Section \ref{n3}. Using our main tool, Theorem \ref{B1}, which was proved in the previous section, in this section we first show monotonicity of $\int|\text{Rm}|^\frac{n}{2}\ dV_t$ and $\int|\Rm|^{\frac{n}{2}\frac{n}{n-2}}\ dV_t$ before proceeding to prove a pointwise bound on $|\text{Rm}|$, which implies the long time existence of the Ricci flow by the blowup alternative of Proposition \ref{shorttime}. Once we have the long-time existence of the flow, we will then complete the proof of Theorem \ref{A} for $n\geq 4$.

\subsection{Long-time existence}\label{Asecsec}

We begin by stating some inequalities for the evolution of $|\Rm|^2$ under the Ricci flow, which follow straightforwardly from the well-known pointwise inequality \eqref{standevo}, using the Kato inequality $|\nabla|T||\leq |\nabla T|$, for any tensor $T$. We have already mentioned a weaker form of \eqref{standevo} in the proof of Lemma \ref{quadid}.

\begin{lemma}\label{formula}
Let $(M^n,g_0)$ be a smooth Riemannian manifold, and let $g(t)$ be a solution of the Ricci flow with initial condition $g_0$. Then we have
\begin{align}
\partial_t|\Rm|^2\leq\Delta|\Rm|^2-2|\nabla\Rm|^2+16|\Rm|^3.\label{standevo}
\end{align}
Suppose also that $(M^n,g_0)$ is an asymptotically flat manifold. Then for $\alpha\geq\max(1,\frac{n}{4})$,
\begin{align}
\frac{d}{dt}\int |\Rm|^{2\alpha}\ dV_t\leq -C_1(\alpha)\int|\nabla|\Rm|^\alpha|^2\ dV_t+C_2(\alpha)\int|\Rm|^{2\alpha+1}\ dV_t,\label{standevo2}
\end{align}
for some constants $C_1(\alpha),C_2(\alpha)>0$.
\end{lemma}
\begin{proof}
As mentioned above, \eqref{standevo} may be found in many places, for instance in \cite{CLN}. Then \eqref{standevo2} is a consequence of \eqref{standevo} and the curvature decay from our assumption of asymptotic flatness; using that $\partial_t\ dV_t=-R\ dV_t$, we obtain
\begin{align}
\frac{d}{dt}\int|\Rm|^{2\alpha}\ dV_t&\leq\alpha\int|\Rm|^{2\alpha-2}\left( \Delta|\Rm|^2  -2|\nabla\Rm|^2+16|\Rm|^3\right)\label{standevo3}
\\
&\quad-\int R|\Rm|^{2\alpha}\ dV_t.\notag
\end{align}
Indeed, since by Proposition \ref{shortbound} $\Rm$ and its derivatives have controlled rates of asymptotic decay on closed intervals where the Ricci flow $g(t)$ exists, the dominated convergence theorem allows us to compute $\frac{d}{dt}\int|\text{Rm}|^{\frac{n}{2}}\ dV_t$ from the corresponding pointwise formula. For the term involving $\Delta|\Rm|^2$ we may integrate by parts on $B_R(x)$ for a fixed $x\in M$ to see that
\begin{align}
    &\int_{B_R(x)}|\Rm|^{2\alpha-2}\Delta|\Rm|^2\ dV_t
    \\
    &=-\int_{B_R(x)}\nabla|\Rm|^{2\alpha-2}\nabla|\Rm|^2\ dV_t+\int_{\partial B_R(x)}|\Rm|^{2\alpha-2}\partial_\nu|\Rm|^2\ dS_t\notag
    \\
    &=-2\int_{B_R(x)}|\Rm|^{2\alpha-2}|\nabla|\Rm||^2 \ dV_t+\int_{\partial B_R(x)}|\Rm|^{2\alpha-2}\partial_\nu|\Rm|^2\ dS_t.\notag
\end{align}
We take $R\rightarrow\infty$ to study this integral, which is the first term on the right in \eqref{standevo3}; again by Proposition \ref{shortbound} we see that the boundary integral tends to zero as $R\rightarrow\infty$. Substituting our result back into \eqref{standevo3}, after straightforward pointwise manipulations of the remaining terms we obtain \eqref{standevo2}.
\end{proof}

\subsubsection{Monotonicity of \texorpdfstring{$\|Rm\|_{L^{\frac{n}{2}}}$}{||Rm||L{n/2}} and a uniform Sobolev inequality}

Now we show that $\int|\Rm|^{\frac{n}{2}}\ dV_t$ is nonincreasing along the Ricci flow under the assumptions of Theorem \ref{A} if it is initially sufficiently small.

\begin{lemma}\label{L2}
Let $(M^n,g_0)$, $n\geq 4$, be an asymptotically flat manifold. There exists a $\delta_1(n)>0$ such that if $\left(\int|\text{Rm}|^{\frac{n}{2}}\ dV_0\right)^{\frac{2}{n}}<\delta_1(n)\frac{1}{C_{g_0}}$, then  $\int|\text{Rm}|^{\frac{n}{2}}\ dV_t$ is nonincreasing along the Ricci flow when it exists. In particular $\int |\text{Rm}|^{\frac{n}{2}}\ dV_t<\delta_1(n)^{\frac{n}{2}}$ along the flow.
\end{lemma}
\begin{proof}
By Lemma \ref{formula} we have
\begin{align}
    &\frac{d}{dt}\int|\Rm|^{\frac{n}{2}}\ dV_t\notag
    \\
    &\leq -C_1\int|\nabla|\Rm|^{\frac{n}{4}}|^2\ dV_t+C_2\int |\Rm|^{\frac{n}{2}+1}\ dV_t\notag
    \\
    &\leq  -C_1\int|\nabla|\Rm|^{\frac{n}{4}}|^2\ dV_t+C_2\left(\int|\text{Rm}|^{\frac{n}{2}}\ dV_t\right)^{\frac{2}{n}}\left(\int|\text{Rm}|^{\frac{n}{4}\frac{2n}{n-2}}\ dV_t\right)^{\frac{n-2}{n}}.\label{L2cont1}
\end{align}
Now we restrict $\delta_1(n)<2$ small enough so that Theorem \ref{B1} holds. Then we find
\begin{align}
-\int|\nabla |\Rm|^{\frac{n}{4}}|^2\ dV_t&\leq-\frac{1}{A}\left(\int|\Rm|^{\frac{n}{4}\frac{2n}{n-2}}\ dV_t\right)^{\frac{n-2}{n}}+\int R^+|\Rm|^{\frac{n}{2}}\ dV_t
\\
&\leq -\frac{1}{A}\left(\int|\Rm|^{\frac{n}{4}\frac{2n}{n-2}}\ dV_t\right)^{\frac{n-2}{n}}+C_3\int |\Rm|^{\frac{n}{2}+1}\ dV_t.\notag
\end{align}
Therefore, continuing from \eqref{L2cont1}, we find that
\begin{align}
&\frac{d}{dt}\int|\Rm|^{\frac{n}{2}}\ dV_t
\\
&\leq -\frac{C_1}{A}\left(\int|\Rm|^{\frac{n}{4}\frac{2n}{n-2}}\ dV_t\right)^{\frac{n-2}{n}}\notag
\\
&\quad +C_2\left(\int|\text{Rm}|^{\frac{n}{2}}\ dV_t\right)^{\frac{2}{n}}\left(\int|\text{Rm}|^{\frac{n}{4}\frac{2n}{n-2}}\ dV_t\right)^{\frac{n-2}{n}}.\notag
\end{align}
Thus, the right-hand-side will be nonpositive if
\begin{align}
    \delta_1(n)\frac{1}{C_{g_0}}\leq\frac{C_1}{C_2}\frac{1}{A}.\label{L2neg}
\end{align}
Since $A$ depends linearly on $C_{g_0}$ we may indeed find a $\delta_1(n)>0$ small enough, depending only on $n$ so that \eqref{L2neg} holds.

This shows that $\frac{d}{dt}\int|\Rm|^{\frac{n}{2}}\ dV_t\leq 0$ whenever $\left(\int|\Rm|^{\frac{n}{2}}\ dV_t\right)^{\frac{2}{n}}<\delta_1(n)\frac{1}{C_{g_0}}$, and since at time $t=0$ we have $\left(\int|\Rm|^{\frac{n}{2}}\ dV_0\right)^{\frac{2}{n}}<\delta_1(n)\frac{1}{C_{g_0}}$ by assumption, it follows that $\int|\Rm|^{\frac{n}{2}}\ dV_t$ is nonincreasing along the Ricci flow.
\end{proof}

As a consequence we can obtain under the same assumptions a uniform, unweighted Sobolev inequality along the Ricci flow we consider in Theorem \ref{A} if we further restrict $\delta_1(n)$ to be sufficiently small.

\begin{corollary}\label{USobolev}
Let $(M^n,g_0)$, $n\geq 4$, be an asymptotically flat manifold. There exists a $\delta_1(n)>0$ such that if $\left(\int|\text{Rm}|^{\frac{n}{2}}\ dV_0\right)^{\frac{2}{n}}<\delta_1(n)\frac{1}{C_{g_0}}$ then the following uniform Sobolev inequality holds along the Ricci flow $g(t)$ whenever it exists, for $u\in W^{1,2}(M)$:
\begin{align}
    \left(\int |u|^{\frac{2n}{n-2}}\ dV_t\right)^{\frac{n-2}{n}}\leq A_0\int |\nabla u|^2\ dV_t\label{usineq}
\end{align}
Moreover, in any fixed dimension $n$, $A_0$ depends linearly on $C_{g_0}$.
\end{corollary}
\begin{proof}
First take $\delta_1(n)>0$ sufficiently small so that the conclusion of Lemma \ref{L2} holds. Then consider the scalar curvature-weighted term in the weighted Sobolev inequality \eqref{WSobolev} of Theorem \ref{B1}. If we further shrink $\delta_1(n)$ we can obtain from the linear dependence of $A$ on $C_{g_0}$ that
\begin{align}
A\int R^+ u^2\ dV_t&\leq A\left(\int |R|^{\frac{n}{2}}\ dV_t\right)^{\frac{2}{n}}\left(\int |u|^{\frac{2n}{n-2}}\ dV_t\right)^{\frac{n-2}{n}}
\\
&\leq\frac{1}{2}\left(\int |u|^{\frac{2n}{n-2}}\ dV_t\right)^{\frac{n-2}{n}}\notag
\end{align}
Absorbing this term back into the left hand side of \eqref{WSobolev}, we obtain the conclusion.
\end{proof}

\subsubsection{Monotonicity of \texorpdfstring{$\|Rm\|_{L^{\frac{n}{2}\frac{n}{n-2}}}$}{}}

Having now established a uniform Sobolev inequality along the Ricci flow of the metric considered in Theorem \ref{A}, we proceed to estimate higher order integral curvature norms. First we will see that the $L^{\frac{n}{2}\frac{n}{n-2}}$ norm of $\Rm$ is nondecreasing following the same argument used to prove Lemma \ref{L2}. Note as in Lemma \ref{L2} that the exponent $\frac{n}{2}\frac{n}{n-2}=\frac{n}{4}\frac{2n}{n-2}$ is exactly the one which appears in Corollary \ref{USobolev} when applied to $u=|\Rm|^{\frac{n}{4}}$. Unlike in Lemma \ref{L2} however, the next lemma will not give control of the explicit size of the $L^{\frac{n}{2}\frac{n}{n-2}}$ norm of $\Rm$, since we did not make any assumption on its size at the initial time $t=0$. 

For the purpose of proving the boundedness of $|\Rm|$ and the long-time existence result of  Proposition \ref{Linfty}, our choice of estimating this particular norm is not particularly special; we could have obtained the same kind of nondecreasing estimate for any fixed $L^p$ with $p>\frac{n}{2}$, and such a control would also suffice. Note however that in order to obtain the monotone nondecreasing of $\int|\Rm|^p\ dV_t$, the smallness of $\delta_1(n)$ in the restriction $\left(\int|\Rm|^{\frac{n}{2}}\ dV_0\right)^{\frac{2}{n}}<\delta_1(n)\frac{1}{C_{g_0}}$ depends on the choice of $p$; thus we cannot directly obtain $L^\infty$ control of $|\Rm|$ by this kind of estimate and require an extra step involving Moser iteration.

In the next section however, for the purpose of proving the long-time curvature decay and convergence of the flow in Section \ref{Afin}, we will see that our choice of estimating the particular $L^q$ norm for $q=\frac{n}{2}\frac{n}{n-2}$ of $|\Rm|$ becomes important.

\begin{lemma}\label{L3}
Let $(M^n,g_0)$, $n\geq 4$, be an asymptotically flat manifold. There exists a $\delta_1(n)>0$ such that if $\left(\int|\text{Rm}|^{\frac{n}{2}}\ dV_0\right)^{\frac{2}{n}}<\delta_1(n)\frac{1}{C_{g_0}}$ then $\int|\text{Rm}|^{\frac{n}{2}\frac{n}{n-2}}\ dV_t$ is nonincreasing along the Ricci flow whenever it exists.
\end{lemma}
\begin{proof}
Again we compute starting from the pointwise formulas of Lemma \ref{formula}:
\begin{align}
\frac{d}{dt}\int|\Rm|^{\frac{n}{2}\frac{n}{n-2}}\ dV_t&\leq -C_1\int|\nabla|\Rm|^{\frac{n}{4}\frac{n}{n-2}}|^2\ dV_t+C_2\int |\Rm|^{\frac{n}{2}\frac{n}{n-2}+1}\ dV_t
\\
&\leq  -C_1\int|\nabla|\Rm|^{\frac{n}{4}\frac{n}{n-2}}|^2\ dV_t\notag
\\
&\quad+C_2\left(\int|\Rm|^{\frac{n}{2}}\ dV_t\right)^{\frac{2}{n}}\left(\int|\Rm|^{\frac{n}{4}\frac{n}{n-2}\frac{2n}{n-2}}\ dV_t\right)^{\frac{n-2}{n}}.\notag
\end{align}
We may now apply Corollary \ref{USobolev} with $u=|\Rm|^{\frac{n}{4}\frac{n}{n-2}}$ and conclude by taking $\delta_1(n)>0$ sufficiently small so that $\frac{d}{dt}\int|\Rm|^{\frac{n}{2}\frac{n}{n-2}}\ dV_t\leq 0$, via an argument similar to that in the proof of Lemma \ref{L2}. Alternatively we could also use Theorem \ref{B1} to argue instead of Corollary \ref{USobolev} since we can deal with the curvature weight by the assumption that $\int|\Rm|^{\frac{n}{2}}\ dV_t$ is small.
\end{proof}

\subsubsection{Bounds on \texorpdfstring{$\|Rm\|_{L^\infty}$}{||Rm||Linfty} and long-time existence}

We now come to the $L^\infty$ estimate of $\Rm$ along the Ricci flow of Theorem \ref{A}. The proof is an adaptation of D. Yang's argument in \cite{DY} from the compact setting to our asymptotically flat setting. Although arguments following this general idea can be found in a variety of sources, we will include details below because our Sobolev inequality from Corollary \ref{USobolev} is slightly different from the one in \cite{DY} and also because in Section \ref{Afin} we will need to refer to the proof in order to show how to it can be used to prove the curvature decay $\sup_{x\in M}t|\Rm(t,x)|\xrightarrow{t\rightarrow\infty}0$.

\begin{proposition}\label{Linfty}
Let $(M^n,g_0)$, $n\geq 4$, be an asymptotically flat manifold. Let $\delta_1(n)>0$ be sufficiently small so that the conclusions of Lemma \ref{L2}, Corollary \ref{USobolev}, and Lemma \ref{L3} hold. Then whenever the Ricci flow exists we have the following estimate for $|\Rm|$:
\begin{align}
    |\Rm|^2\leq \max( C_1(g_0) t^{-\frac{2(n-2)}{n}}, C_2(g_0)t^{\frac{4(n-2)}{n^2}}),
\end{align}
for constants $C_1(g_0)$ and $C_2(g_0)$ depending only on the initial metric $g_0$. In particular, they depend only on the initial Sobolev constant $C_0$ and the initial curvature quantity $\int|\Rm|^{\frac{n}{2}\frac{n}{n-2}}\ dV_0$.
\end{proposition}
\begin{proof}
For convenience of notation we let $f(t,x)=|\Rm|^2$. Again, the decay estimates for curvature resulting from Proposition \ref{shorttime} will allow us to integrate by parts and differentiate in the integral below using the pointwise formulas of Lemma \ref{formula}. For $p\geq\frac{n}{4}$ and any fixed $q>\frac{n}{2}$, we have
\begin{align}
\frac{1}{p}\frac{d}{dt}\int f^p\ dV_t
&\leq\int f^{p-1}\Delta f+16 |\Rm|f^p\ dV_t
\\
&\leq-\frac{4(p-1)}{p^2}\int|\nabla(f^{\frac{p}{2}})|^2\ dV_t\notag
\\
&\quad+16\left(\int |\Rm|^q\ dV_t\right)^{\frac{1}{q}}\left(\delta^{-\frac{n}{2q}}\int f^p\ dV_t\right)^{1-\frac{n}{2q}}\notag
\\
&\qquad\left(\delta^{\left(1-\frac{n}{2q}\right)\frac{n}{n-2}}\int f^{\frac{p}{2}\frac{2n}{n-2}}\ dV_t\right)^{\frac{n-2}{n}\frac{n}{2q}},\notag
\end{align}
where the last line holds for any $\delta>0$. Let $\beta>0$ be a constant such that we have the bound $16\|\Rm\|_{L^q}\leq\beta$ along our Ricci flow. We set $q=\frac{n}{2}\frac{n}{n-2}$ so that by Lemma \ref{L3} there indeed exists such a $\beta$. Continuing, we obtain,
\begin{align}
&\frac{1}{p}\frac{d}{dt}\int f^p\ dV_t
\\
&\leq-\frac{4(p-1)}{p^2}\int|\nabla(f^{\frac{p}{2}})|^2+\beta\delta^{-\frac{n}{2q}}\int f^p\ dV_t+\beta\delta^{1-\frac{n}{2q}}\left(\int f^{\frac{p}{2}\frac{2n}{n-2}}\ dV_t\right)^{\frac{n-2}{n}}.\notag
\end{align}
We apply the Sobolev inequality of Corollary \ref{USobolev} and set $\delta=\left(\frac{3p-4}{\beta A_0 p^2}\right)^{\frac{2q}{2q-n}}$ to obtain
\begin{align}
\frac{d}{dt}\int f^p\ dV_t+\int|\nabla(f^{\frac{p}{2}})|^2\ dV_t\leq C_{p,q,\beta}\int f^p\ dV_t,\label{Cstep1}
\end{align}
where $C_{p,q,\beta}=p\beta\left(\frac{\beta A_0 p^2}{3p-4}\right)^{\frac{n}{2q-n}}$. Unlike other constants such as $C_1$, $C_2$ that we have considered earlier, $C_{p,q,\beta}$ will always denote this particular value given $p,q,\beta$. Now let $T>0$ be such that the Ricci flow is defined on $[0,T]$ and  define, for $0<\tau<\tau'<T$, the function $\psi:[0,T]\rightarrow[0,1]$,
\begin{align}
\psi(t)=
\begin{cases}
0,\quad &0\leq t\leq\tau,
\\
\frac{t-\tau}{\tau'-\tau},\quad &\tau\leq t\leq\tau',
\\
1,\quad&\tau'\leq t\leq T.
\end{cases}
\end{align}
Then we multiply \eqref{Cstep1} by $\psi$ and find that
\begin{align}
\frac{d}{dt}\left(\psi\int f^p\ dV_t\right)+\psi\int|\nabla(f^{\frac{p}{2}})|^2\ dV_t\leq \left(C_{p,q,\beta}\psi+\psi'\right)\int f^p\ dV_t,
\end{align}
so that integrating, for any $\tilde{t}\in[\tau',T]$ we have
\begin{align}
\int f^p\ dV_{\tilde{t}}+\int_{\tau'}^{\tilde{t}}\int|\nabla(f^{\frac{p}{2}})|^2\ dV_t\ dt\leq \left(C_{p,q,\beta}+\frac{1}{\tau'-\tau}\right)\int_\tau^{T}\int f^p\ dV_t\ dt.\label{Cstep2}
\end{align}
We define for $p\geq \frac{n}{4}$ and $\tau\in[0,T]$,
\begin{align}
H(p,\tau)=\int_\tau^{T}\int f^p\ dV_t\ dt,
\end{align}
and let $\nu =1+\frac{2}{n}$. We now claim that for $p\geq \frac{n}{4}$ and $0\leq\tau<\tau'\leq T$, 
\begin{align}
H(\nu p,\tau')\leq A_0\left(C_{p,q,\beta}+\frac{1}{\tau'-\tau}\right)^\nu H(p,\tau)^\nu.\label{Cstep3}
\end{align}
Indeed,
\begin{align}
\int_{\tau'}^{T}\int f^{\nu p}\ dV_t\ dt&\leq\int_{\tau'}^{T}\left(\int f^p\ dV_t\right)^{\frac{2}{n}}\left(\int f^{\frac{p}{2}\frac{2n}{n-2}}\ dV_t\right)^{\frac{n-2}{n}}\ dt
\\
&\leq A_0 \left(\sup_{\tau'\leq t\leq T}\int f^p\ dV_t\right)^{\frac{2}{n}}\int_{\tau'}^{T}\int|\nabla (f^{\frac{p}{2}})|^2\ dV_t\ dt,\notag
\end{align}
so that \eqref{Cstep2} implies the claim. Now we iterate \eqref{Cstep3} to obtain $L^\infty$ control. Let $p_0=\frac{q}{2}=\frac{n}{4}\frac{n}{n-2}$, and define
\begin{align}
\eta=\nu^{\frac{2q}{2q-n}},\quad p_k=\nu^k p_0,\quad\tau_k=(1-\eta^{-k})T,\quad \Phi_k=H(p_k,\tau_k)^{\frac{1}{p_k}}.
\end{align}
We apply \eqref{Cstep3} to see that
\begin{align}
\Phi_{k+1}&=H(\nu p_k,\tau_{k+1})^{\frac{1}{\nu p_k}}
\\
&\leq A_0^{\frac{1}{\nu p_k}}\left(C_{p_k,q,\beta}+\frac{1}{\tau_{k+1}-\tau_k}\right)^{\frac{1}{p_k}}H(p_k,\tau_k)^{\frac{1}{p_k}}\notag
\\
&\leq A_0^{\frac{1}{\nu p_k}}\left(A_0^{\frac{n}{2q-n}}\left(\beta p_0\right)^{\frac{2q}{2q-n}}\eta^k+C(n,q)\frac{\eta^{k+1}}{T}\right)^{\frac{1}{p_k}}H(p_k,\tau_k)^{\frac{1}{p_k}}\notag
\\
&\leq A_0^{\frac{1}{\nu p_k}}\left(A_0^{\frac{n}{2q-n}}\left(\beta p_0\right)^{\frac{2q}{2q-n}}+C(n,q)\frac{\eta}{T}\right)^{\frac{1}{p_k}}\eta^{\frac{k}{p_k}}\Phi_k.\notag
\end{align}
Now since $\sum_{i=0}^\infty\frac{1}{p_k}=\frac{2(n^2-4)}{n^2}<\infty$ and $\sum_{i=0^\infty}\frac{k}{p_k}<\infty$, we can take $k\rightarrow\infty$ to obtain
\begin{align}
\sup_{x\in M}|\Rm(T,x)|^2&\leq\max(C_1(A_0,\beta)T^{-\frac{2(n^2-4)}{n^2}},C_2(A_0,\beta))\Phi_0
\\
&\leq \max(C_1(A_0,\beta)T^{-\frac{2(n^2-4)}{n^2}},C_2(A_0,\beta))\beta^2 T^{\frac{4(n-2)}{n^2}},\notag
\end{align}
which gives us the desired conclusion, since the choice of $T$ for which the Ricci flow exists on $[0,T]$ was arbitrary.
\end{proof}

\begin{corollary}\label{longtimebound}
Let $(M^n,g_0)$, $n\geq 4$, be an asymptotically flat manifold. Let $\delta_1(n)>0$ be sufficiently small so that the conclusions of Lemma \ref{L2}, Corollary \ref{USobolev}, and Lemma \ref{L3} hold. Then the Ricci flow $g(t)$ exists for all $t\in[0,\infty)$ and there exists a $K>0$ such that
\begin{align}
\sup_{\substack{t\in[0,\infty)\\x\in M}}|\Rm|\leq K<\infty.
\end{align}
\end{corollary}
\begin{proof}
The bounds we have obtained in Proposition \ref{Linfty} immediately imply that the flow exists for all times $t\in[0,\infty)$, by the blowup alternative in the short-time existence statement of Proposition \ref{shorttime}. 

For the uniform bound on $|\Rm|$ observe first that by the estimates in Proposition \ref{Linfty} $|\Rm|$ is bounded for $t\in[1,2]$. Moreover the the estimates on $|\Rm|$ in the short-time existence results of Proposition \ref{shorttime} tells us that $|\Rm|$ is bounded on $[0,1]$. Next, by Proposition \ref{USobolev} we have a uniform upper bound on the Sobolev constant $C_{g(t)}$ along the flow, and by Proposition \ref{L3}, $\int|\Rm|^{\frac{n}{2}\frac{n}{n-2}}\ dV_t$ is nonincreasing along the flow. Therefore we can apply the estimates of Proposition \ref{Linfty} with an initial time $t=1$ instead of $t=0$ to see that
\begin{align}
\sup_{\substack{t\in[2,3]\\x\in M}}|\Rm|\leq \sup_{\substack{t\in[1,2]\\x\in M}}|\Rm|.
\end{align}
Repeating this by applying Proposition \ref{Linfty} for initial times $t\in\mb{N}$, we conclude that
\begin{align}
\sup_{\substack{t\in[0,\infty)\\x\in M}}|\Rm|\leq \sup_{\substack{t\in[0,2]\\x\in M}}|\Rm|<\infty.
\end{align}
\end{proof}
Next we will look more carefully at $\int|\Rm|^{\frac{n}{2}\frac{n}{n-2}}\ dV_t$ and take advantage of the improved decay that we obtain to pass to decay of $|\Rm|$ in time.

\subsection{Curvature decay and convergence --- proof of Theorem \ref{A}}\label{Afin}

We will now complete the proof of Theorem \ref{A} when $n\geq 4$ by showing that the long-time existence of the flow considered in Theorem \ref{A} along with our estimates in Section \ref{Asecsec} imply its convergence to $\mb{R}^n$ with the standard flat metric. Below we first prove the curvature decay estimate $\|Rm_{g(t)}\|_{L^\infty}=o(t^{-1})$, and then discuss why this decay estimate is sufficient to conclude convergence.

\subsubsection{Decay of \texorpdfstring{$\|Rm\|_{L^{\frac{n}{2}\frac{n}{n-2}}}$}{||Rm||L(n/2)(n/(n-2))}}

As a preliminary step, we will first obtain decay of the curvature in an integral sense along the flow, which improves the monotonic nonincreasing property established in Lemma \ref{L3}.

\begin{lemma}\label{L4decay}
Let $(M^n,g_0)$, $n\geq 4$, be an asymptotically flat manifold. Let $\delta_1(n)>0$ be sufficiently small so that the conclusions of Lemma \ref{L2}, Corollary \ref{USobolev}, and Lemma \ref{L3} hold. Then for every $\eta>0$, there exists $T_\eta>0$ such that along the Ricci flow,
\begin{align}
\left(\int|\Rm|^{\frac{n}{2}\frac{n}{n-2}}\ dV_t\right)^{\frac{n-2}{n}}\leq\eta t^{-1},\quad\text{for all }t\geq T_\eta.\label{L4dest}
\end{align}
\end{lemma}
\begin{proof}
We showed in the proof of Lemma \ref{L2} that there exists a $C=C(g_0)>0$ such that at all times where our Ricci flow is defined,
\begin{align}
\frac{d}{dt}\int|\Rm|^{\frac{n}{2}}\ dV_t+C\left(\int|\Rm|^{\frac{n}{2}\frac{n}{n-2}}\ dV_t\right)^{\frac{n-2}{n}}\leq 0.\label{L2evolve}
\end{align}
Furthermore we know by Corollary \ref{longtimebound} that the flow exists for all times $t\in[0,\infty)$. Since $\int|\Rm|^{\frac{n}{2}}\ dV_t$ is nonnegative, we can integrate \eqref{L2evolve} to see that
\begin{align}
    \int_0^\infty \left(\int|\Rm|^{\frac{n}{2}\frac{n}{n-2}}\ dV_t\right)^{\frac{n-2}{n}}\ dt<\infty.
\end{align}
This combined with the fact that $\left(\int|\Rm|^{\frac{n}{2}\frac{n}{n-2}}\ dV_t\right)^{\frac{n-2}{n}}$ is monotonic nonincreasing along the flow by Lemma \ref{L3} implies that
\begin{align}
\lim_{t\rightarrow\infty}t\left(\int|\Rm|^{\frac{n}{2}\frac{n}{n-2}}\ dV_t\right)^{\frac{n-2}{n}}=0,
\end{align}
which gives \eqref{L4dest}.
\end{proof}

\subsubsection{Decay of \texorpdfstring{$\|Rm\|_{L^\infty}$}{||Rm||Linfty} from decay of \texorpdfstring{$\|Rm\|_{L^{\frac{n}{2}\frac{n}{n-2}}}$}{||Rm||L(n/2)(n/(n-2))}}

Using the integral curvature decay estimate just obtained we can now revisit the iteration argument of Proposition \ref{Linfty} to prove our $L^\infty$ decay estimate for $Rm$.

\begin{proposition}\label{supdec}
Let $(M^n,g_0)$, $n\geq 4$, be an asymptotically flat manifold. Let $\delta_1(n)>0$ be sufficiently small so that the conclusions of Lemma \ref{L2}, Corollary \ref{USobolev}, and Lemma \ref{L3} hold. Then the singularity at infinity of the Ricci flow $g(t)$ starting from $g_0$ which exists for all times $t\in[0,\infty)$ by Corollary \ref{longtimebound} satisfies
\begin{align}
\lim_{t\rightarrow\infty}t\sup_{x\in M}|\Rm(t,x)|=0.\label{T3}
\end{align}
\end{proposition}
\begin{proof}
We will follow a line of argument similar to that used in the proof of Proposition \ref{Linfty}. Recall that in that proof we used a constant $\beta>0$ such that $\|\Rm\|_{L^q}\leq\beta$, where $q=\frac{n}{2}\frac{n}{n-2}$.  But now by Lemma \ref{L4decay} we have a better estimate --- given $\eta>0$,
\begin{align}
\|\Rm(t)\|_{L^q}=\left(\int|\Rm|^{\frac{n}{2}\frac{n}{n-2}}\ dV_t\right)^{\frac{2}{n}\frac{n-2}{n}}\leq\left(\eta t^{-1}\right)^{\frac{2}{n}},\quad\text{for all } t\geq T_\eta.\label{betterest}
\end{align}

Recall that in the proof of Proposition \ref{Linfty} we produced an $L^\infty$ estimate of $\Rm$ at time $t=T$ using an $L^q$ bound of $\Rm$ at time $t=0$. Here, taking advantage of the uniform Sobolev inequality \eqref{usineq} along the flow combined with \eqref{betterest}, we will instead produce an $L^\infty$ estimate of $\Rm$ at time $t=T$ using $L^q$ bounds of $\Rm$ for $t\in[\frac{T}{2},T]$, given any $T\geq 2T_\eta$. Define now $\beta(t)=16\left(\eta t^{-1}\right)^{\frac{2}{n}}$. Then we have the following analogue of \eqref{Cstep3}:
\begin{align}
H(\nu p,\tau')\leq A_0\left(C_{p,q,\beta(\tau)}+\frac{1}{\tau'-\tau}\right)^\nu H(p,\tau)^\nu,\label{Cstepnew}
\end{align}
which holds for any $\frac{T}{2}\leq\tau<\tau'\leq T$. Again we intend to iterate to obtain $L^\infty$ control. Define as before $p_0=\frac{q}{2}=\frac{n}{4}\frac{n}{n-2}$ with
\begin{align}
\eta=\nu^{\frac{2q}{2q-n}},\quad p_k=\nu^k p_0,\quad\Phi_k=H(p_k,\tau_k)^{\frac{1}{p_k}},
\end{align}
but this time with $\tau_k=\frac{T}{2}+(1-\eta^{-k})\frac{T}{2}$. We then see that (restricting to $\eta<1$)
\begin{align}
\Phi_{k+1}&\leq A_0^{\frac{1}{p_k}}\left(A_0^{\frac{n}{2q-n}}\left(\beta(\tau_k) p_0\right)^{\frac{2q}{2q-n}}+C(n,q)\frac{\eta}{T}\right)^{\frac{1}{p_k}}\eta^{\frac{k}{p_k}}\Phi_k
\\
&\leq A_0^{\frac{1}{p_k}}\left(A_0^{\frac{n}{2q-n}}p_0^{\frac{2q}{2q-n}}\frac{C_1}{T}+C(n,q)\frac{\eta}{T}\right)^{\frac{1}{p_k}}\eta^{\frac{k}{p_k}}\Phi_k.\notag
\end{align}
Iterating as in the proof of Proposition \ref{Linfty} we therefore see that for some constant $C>0$ depending only on $(M^n,g_0)$,
\begin{align}
\sup_{x\in M}|\Rm(T,x)|^2\leq CT^{-\frac{2(n^2-4)}{n^2}}\Phi_0.\label{pen1}
\end{align}
Finally, we have that
\begin{align}
\Phi_0&=\left(\int_{\frac{T}{2}}^T\int|\Rm|^{\frac{n}{2}\frac{n}{n-2}}\ dV_t\ dt\right)^{\frac{4}{n}\frac{n-2}{n}}
\\
&\leq \left(\int_{\frac{T}{2}}^T (\eta t^{-1})^{\frac{n}{n-2}}\ dt\right)^{\frac{4}{n}\frac{n-2}{n}}\notag
\\
&\leq C(n) \eta^{\frac{4}{n}}T^{-\frac{8}{n^2}}.\notag
\end{align}
Combining this with \eqref{pen1}, we conclude \eqref{T3}, since we find that for some constant $C$ depending only on $(M^n,g_0)$ and any $\eta>0$,
\begin{align}
\sup_{x\in M}|\Rm(T,x)|^2\leq C \eta^{\frac{4}{n}}T^{-2},\quad\text{for all }T\geq 2T_\eta.
\end{align}
\end{proof}

\subsubsection{Sufficient decay of \texorpdfstring{$\|Rm\|_{L^\infty}$}{||Rm||Linfty} implies convergence to flat space}

Now that we have proved the $L^\infty$ curvature decay estimate of Proposition \ref{supdec}, a straightforward adaptation of the results of \cite[Sections 4--5]{YLi} implies the weighted convergence of $g(t)$ to an AF metric on $M$ with $|\Rm|\equiv 0$, so that $M$ is diffeomorphic to $\mb{R}^n$. First we now quote the particular statement due to \cite{YLi} which we will adapt:

\begin{proposition}[{\cite[Sections 4-5]{YLi}}]\label{Lisubt2}
    Let $(M^n,g_0)$ be an asymptotically flat manifold of order $\tau>0$. Suppose that the Ricci flow $g(t)$ with initial condition $g(0)=g_0$ satisfying $R_{g_0}\geq 0$ exists for all times $0\leq t<\infty$ and moreover has curvature decay satisfying
    \begin{align}
        \lim_{t\rightarrow\infty}t\sup_{x\in M}|\Rm(t,x)|=0.
    \end{align}
    Then there exists an asymptotically flat metric $g(\infty)$ with the same asymptotically flat coordinate system as the metrics $g(t)$ such that the $g(t)$ converge to $g(\infty)$ as $t\rightarrow\infty$ in $C_{-\tau'}^\infty(M)$, for any $\tau'\in(0,\min(\tau,n-2))$.
\end{proposition}

Note however that unlike in Proposition \ref{Lisubt2} along with in the rest of Li's paper, we have not made any assumptions on the sign of the scalar curvature of the initial metric. Hence that result cannot be directly applied to our setting. Instead, below we will describe how the arguments used to prove Proposition \ref{Lisubt2} in \cite{YLi} can be modified so that the assumption $R_{g_0}\geq 0$ can be replaced by the assumptions of Theorem \ref{A}. As a result this concludes the proof of Theorem \ref{A} in dimensions $n\geq 4$.

\begin{proposition}\label{Lidapt}
    Let $(M^n,g_0)$ be an asymptotically flat manifold of order $\tau>0$. Suppose that the Ricci flow $g(t)$ with initial condition $g(0)=g_0$ exists for all times $0\leq t<\infty$ and moreover has curvature decay satisfying
    \begin{align}
        \lim_{t\rightarrow\infty}t\sup_{x\in M}|\Rm(t,x)|=0.
    \end{align}
    There exists $\delta_1(n)>0$ such that if $\left(\int|\Rm|^{\frac{n}{2}}\ dV_0\right)^{\frac{2}{n}}<\delta_1(n)\frac{1}{C_{g_0}}$, then there exists a flat, asymptotically flat metric $g(\infty)$ with the same asymptotically flat coordinate system as the metrics $g(t)$ such that the $g(t)$ converge to $g(\infty)$ as $t\rightarrow\infty$ in $C_{-\tau'}^\infty(M)$, for any $\tau'\in(0,\min(\tau,n-2))$.
\end{proposition}
\begin{proof}
We refer to \cite[Sections 4-5]{YLi} for most parts of the proof. There it is assumed $R\geq 0$ along the flow; however, the first part of the proof of \cite[Theorem 4.4]{YLi} is the only point where $R\geq 0$ is used. At that point, we have a solution of the heat equation $\partial_t u=\Delta u$ satisfying on any $[0,T]$ the bounds $c_1(T) r^{-2-\tau}\leq u(t,x)\leq c_2(T) r^{-2-\tau}$. Here $r$ is a positive function on $M$, equal to $|x|$ on the image of the asymptotically flat coordinate system in $\mb{R}^n$. We just need to show that $\frac{d}{dt}\int u^p dV_t\leq 0$ for any fixed $p\in(\frac{n}{2+\min(\tau,n-2)},\frac{n}{2})$ if we assume that $\int|Rm|^{\frac{n}{2}}\ dV_0$ is small, instead of that $R\geq 0$ along the flow.

To see this, by Lemma \ref{L2}, Corollary \ref{USobolev}, and Lemma \ref{L3} we have that
\begin{align}
    \int|\Rm|^{\frac{n}{2}}\ dV_t\leq \int|\Rm|^{\frac{n}{2}}\ dV_0<\left(\delta(n)\frac{1}{C_{g_0}}\right)^{\frac{n}{2}},
\end{align}
if $\left(\int|\Rm|^{\frac{n}{2}}\ dV_0\right)^{\frac{2}{n}}<\delta(n)\frac{1}{C_{g_0}}$. We then observe that
\begin{align}
    \frac{d}{dt}\int u^p\ dV_t&\leq -\frac{4(p-1)}{p}\int|\nabla u^{\frac{p}{2}}|^2\ dV_t-\int R u^p\ dV_t\label{zeron2}
    \\
    &\leq -\frac{4(p-1)}{p}\frac{1}{A}\left(\int u^{\frac{p}{2}\frac{2n}{n-2}}\ dV_t\right)^{\frac{n-2}{n}}\notag
    \\
    &\quad+\left(\int|Rm|^{\frac{n}{2}}\right)^{\frac{2}{n}}\left(\int u^{\frac{p}{2}\frac{2n}{n-2}}\ dV_t\right)^{\frac{n-2}{n}}\notag
\end{align}
which is indeed nonpositive if $\left(\int|\Rm|^{\frac{n}{2}}\ dV_0\right)^{\frac{2}{n}}<\delta_1(n)\frac{1}{C_{g_0}}$ with a possibly smaller value of $\delta_1(n)$ than appeared in earlier lines, if we restrict to $p\in(\frac{n}{2+\min(\tau,n-3)},\frac{n}{2})$, since this bounds the coefficient $-\frac{4(p-1)}{p}$ away from zero. As remarked earlier, after this step the rest of the proof in \cite{YLi} does not use that $R\geq 0$ along the flow and so by the exact same arguments we obtain the $C^\infty_{-\tau'}(M)$ convergence of $g(t)$ for $\tau'\in(0,\min(\tau,n-3))$.

However, if $\tau> n-3$ then we have not yet realized the $C^\infty_{-\tau'}(M)$ convergence of $g(t)$ for the full claimed range of $\tau'\in(0,\min(\tau,n-2))$. To achieve this, we need to see that $\frac{d}{dt}\int u^p dV_t\leq 0$ also for $p\in(\frac{n}{2+\min(\tau,n-2)},\frac{n}{2+\min(\tau,n-3)}]$. For this, since now $g(t)\xrightarrow{t\rightarrow\infty} g(\infty)$ in $C^\infty_{-\tau'}(M)$ for some $\tau'>0$, the monotonic nonincreasing quantity $\int|\Rm|^{\frac{n}{2}}\ dV_t$ goes to $0$ as $t\rightarrow\infty$. Returning to \eqref{zeron2} we therefore see that for any $p>\frac{n}{2+\min(\tau,n-2)}$ we have $\frac{d}{dt}\int u^p dV_t\leq 0$ when $t\geq T_p$ is sufficiently large. Then we can apply again the remaining arguments from \cite{YLi}, starting at $g(T_p)$ instead of $g(0)$, to conclude the $C^\infty_{-\tau'}(M)$ convergence of $g(t)$ for the full range $\tau'\in(0,\min(\tau',n-2))$.
\end{proof}

\section{The case \texorpdfstring{$n=3$}{n=3}}\label{n3}

In this section we will complete the proof of Theorem \ref{A} by addressing the case of dimension $n=3$. Here, a technical issue arises in estimating $\partial_t\int|Rm|^{\frac{n}{2}}\ dV_t$ as we did in Section \ref{Asec}, since up to now we have computed everything using the differential inequality satisfied by the smooth function $|Rm|^2$. When $n=3$ however, we would have
\begin{align}
    \partial_t|Rm|^{\frac{3}{2}}=\frac{3}{4}|Rm|^{-\frac{1}{2}}\partial_t|Rm|^2.
\end{align}
But $|Rm|$ may be vanishing at some points along the flow, so this identity creates difficulties if we try to establish an analogue of Lemma \ref{formula} in dimension $n=3$ via integration by parts.

To avoid this issue we instead proceed in the following way. Given an AF manifold satisfying the hypotheses of Theorem \ref{A}, for $\epsilon>0$ let $u_\epsilon(t,x)$ be the solution of the heat equation,
\begin{align}
    \begin{cases}
    \partial_t u_\epsilon=\Delta u_\epsilon,
    \\
    u_\epsilon(0,x)
    =\epsilon r^{-2-\tau}.
    \end{cases}
\end{align}
Recall that $r$ is taken to be a positive function on $M$, equal to $|x|$ on the image of the asymptotically flat coordinate system in $\mb{R}^n$. In particular, $u_\epsilon(0,x)$ satisfies the same spatial decay rate as $|Rm_{g_0}|$. Moreover, by standard maximum principle arguments as in \cite{DM,YLi}, we have on any interval $[0,T]$ on which the Ricci flow of $(M^n,g_0)$ exists that
\begin{align}
    &\epsilon C_1(T) r^{-2-\tau}\leq u_\epsilon(t,x)\leq \epsilon C_2(T) r^{-2-\tau},\notag
    \\
    &|\nabla u_\epsilon(t,x)|\leq \epsilon C_2(T) r^{-3-\tau},\quad 
    |\nabla^2 u_\epsilon(t,x)|\leq \epsilon C_2(T) r^{-4-\tau}.\notag
\end{align}
Furthermore, we note that $\partial_t u_\epsilon^2=\Delta u_\epsilon^2-2|\nabla u_\epsilon|^2$. Combining this with \eqref{standevo} we obtain
\begin{align}
    \partial_t (|Rm|^2+u_\epsilon^2)&\leq\Delta(|Rm|^2+u_\epsilon^2)-2|\nabla Rm|^2-2|\nabla u_\epsilon|^2+16|Rm|^3.\label{n3ineq}
\end{align}
Since $|\nabla(|Rm|^2+u_\epsilon^2)|^2\leq 4(|Rm|^2+u_\epsilon^2)\left(|\nabla|Rm||^2+|\nabla u_\epsilon|^2\right)$, then setting $|Rm|^2+u_\epsilon^2=K_\epsilon^2$ we find that
\begin{align}
    \partial_t K_\epsilon^2\leq\Delta K_\epsilon^2-2|\nabla K_\epsilon|^2+16 K_\epsilon^3.
\end{align}
Note also that this differential inequality satisfied by $K_\epsilon$ takes the same form as the differential inequality satisfied by \eqref{standevo}, and moreover $K_\epsilon\geq \epsilon C_1(T) r^{-2-\tau}$ is bounded away from $0$ by a function with a controlled rate of spatial decay on any time interval $[0,T]$ on which the Ricci flow exists.  Therefore we have an analogue of Lemma \ref{formula} for $K_\epsilon$.

\begin{lemma}\label{n3formula}
Let $(M^n,g_0)$, $n=3$, be an asymptotically flat manifold. Then for $\alpha\geq \frac{3}{4}$,
\begin{align}
\frac{d}{dt}\int K_\epsilon^{2\alpha}\ dV_t\leq -C_1(\alpha)\int|\nabla K_\epsilon^\alpha|^2\ dV_t+C_2(\alpha)\int K_\epsilon^{2\alpha+1}\ dV_t,
\end{align}
for some constants $C_1(\alpha),C_2(\alpha)>0$.
\end{lemma}
\begin{proof}
Since $K_\epsilon$ is strictly positive (unlike $|Rm|$) we can compute using the spatial asymptotic upper and lower bounds of $u_\epsilon$ that
\begin{align}
    \frac{d}{dt}\int K_\epsilon^{2\alpha}\ dV_t&\leq\alpha\int K_\epsilon^{2\alpha-2}\left(\Delta K_\epsilon^2-2|\nabla K_\epsilon|^2+16 K_\epsilon^3\right)\ dV_t-\int R K_\epsilon^{2\alpha}\ dV_t\notag
    \\
    &\leq \alpha\int \left(2-4\alpha\right)K_\epsilon^{2\alpha-2}|\nabla K_\epsilon|^2\ dV_t+C_2(\alpha)\int K_\epsilon^{2\alpha+1}\ dV_t\notag
    \\
    &\leq -C_1(\alpha)\int|\nabla K_\epsilon^{\alpha}|^2\ dV_t+C_2(\alpha)\int K_\epsilon^{2\alpha+1}\ dV_t.\notag
\end{align}
\end{proof}
Having established Lemma \ref{n3formula}, we may now take $\epsilon>0$ sufficiently small so that under the hypotheses of Theorem \ref{A} we also have $\left(\int K_\epsilon^{\frac{n}{2}}\ dV_0\right)^{\frac{n}{2}}<\delta(n)C_{g_0}$, and then we can follow the idea of the proof in Section \ref{Asec} applied to the function $K_\epsilon$ to complete the proof of Theorem \ref{A} by proving it in the case $n=3$.

We will now briefly describe how this proceeds. At this point in Lemma \ref{n3formula} we have found an analogue for Lemma \ref{formula} for dimension $n=3$, with $K_\epsilon$ replacing the role of $|Rm|$. From this and the fact that $K_\epsilon>|Rm|$ we straightforwardly obtain the analogue of Lemma \ref{L2} for dimension $n=3$, again with $|Rm|$ replaced by $K_\epsilon$. Thus Corollary \ref{USobolev} follows as before, and we may continue to obtain analogues of Lemma \ref{L3} and Proposition \ref{Linfty}, still for dimension $n=3$ and with $|Rm|$ replaced by $K_\epsilon$. Since $K_\epsilon>|Rm|$ we therefore obtain $L^\infty$ bounds for $|Rm|$ as well, and also the long-time existence result in Corollary \ref{longtimebound} for the Ricci flow of $(M^3,g_0)$. To pass from long-time existence of the flow to convergence, we first obtain the analogues of Lemma \ref{L4decay} and Proposition \ref{supdec} for dimension $n=3$ by another use of Lemma \ref{n3formula}, $K_\epsilon$ replacing $|Rm|$. Thus,
\begin{align}
    \lim_{t\rightarrow\infty}t\sup_{x\in M}K_\epsilon=0,
\end{align}
which implies $\lim_{t\rightarrow\infty}t\sup_{x\in M}|\Rm(t,x)|=0$. To conclude we apply Proposition \ref{Lidapt}, which can be used in our setting of dimension $n=3$ since by now we have shown that $\left(\int K_\epsilon^{\frac{n}{2}}\ dV_t\right)^{\frac{n}{2}}<\delta(n)C_{g_0}$ remains small and bounded along the flow under the hypotheses of Theorem \ref{A}.

\begin{appendices}
\section{Conformally flat AF metrics}
We will discuss here the special case of asymptotically flat metrics $g_0=e^{2u}|dx|^2$ on $\mb{R}^n$ which are conformal deformations of the flat metric. We begin by giving the proof of Corollary \ref{cfc} as a consequence of Theorem \ref{A}.

\subsection{Proof of Corollary \ref{cfc}}

\begin{proof}[Proof of Corollary \ref{cfc}]
By the conformal invariance of the Yamabe quotient, we have that for every $u\in W^{1,2}(\mb{R}^n,g_0)$,
\begin{align}
    \left(\int|u|^{\frac{2n}{n-2}}\ dV_0\right)^{\frac{n-2}{n}}\leq C_{n,e}\left(\int|\nabla u|^2+\frac{1}{6} Ru^2\ dV_0\right),
\end{align}
where $C_{n,e}$ is the Sobolev constant for the flat metric on $\mb{R}^n$. Therefore if $\int|Rm|^{\frac{n}{2}}\ dV_0$ is sufficiently small then we find that the Sobolev constant of $g_0$ satisfies $C_{g_0}\leq 2C_{n,e}$. Moreover by Theorem \ref{A} there exists an $\epsilon=\epsilon(2 C_{n,e})>0$ such that if we also have $\int|Rm|^{\frac{n}{2}}\ dV_0<\epsilon$, then the Ricci flow starting from $g_0$ will converge to the flat metric on $\mb{R}^n$. Thus we can find $\Lambda(n)>0$ as claimed.
\end{proof}

Note that in this setting we do not require a separate assumption on the Sobolev constant because of the conformal invariance of the Yamabe constant. As mentioned earlier, Corollary \ref{cfc} applies in particular to rotationally symmetric asymptotically flat metrics on $\mb{R}^n$, and therefore gives another long-time existence and convergence statement for the Ricci flow in this setting different from that in \cite{OW}. The main result there states that the Ricci flow of a rotationally symmetric, asymptotically flat metric $g_0$ on $\mb{R}^n$ exists for all times $t\in[0,\infty)$ and converges to the flat metric on $\mb{R}^n$ in the pointed Cheeger-Gromov sense if the initial metric $(\mb{R}^n,g_0)$ contains no minimal hyperspheres. For clarity we quote the relevant parts of that result below.
\begin{theorem}[{\cite[Theorem 1.1]{OW}}]\label{OWt}
    Let $g_0$ be an asymptotically flat, rotationally symmetric metric on a fixed coordinate system on $\mb{R}^n$, with $n\geq 3$. If $(\mb{R}^n,g_0)$ does not contain any minimal hyperspheres, then there exists a solution $g(t,x)\in C^\infty((0,\infty)\times\mb{R}^n)$ of the Ricci flow with initial condition $g(0)=g_0$ which remains asymptotically flat for all times $t\in[0,\infty)$ and:
    \begin{enumerate}[(a)]
        \item We have $g(t,x)\in C^1([0,T]\times\mb{R}^n)$,
        \item For each integer $\ell\geq 0$ there exists a constant $C_\ell>0$ such that
        \begin{align}
            \sup _{x \in \mathbb{R}^{n}}\left|\nabla^{\ell} \operatorname{Rm}(t, x)\right|_{g(t, x)} \leq \frac{C_{\ell}}{(1+t) t^{\ell / 2}},\quad\text{for all }t>0,\notag
        \end{align}
        \item The flow converges to flat Euclidean space in the pointed Cheeger-Gromov sense as $t\rightarrow\infty$.
    \end{enumerate}
\end{theorem}

By proving Theorem \ref{C1} below, we will observe the following relation between Corollary \ref{cfc} and Theorem \ref{OWt}: the pinching assumption $\int|\Rm|^{\frac{n}{2}}\ dV_g<\Lambda(n)$ which we assumed in Corollary \ref{cfc} also rules out minimal hyperspheres in the rotationally symmetric setting (in fact without requiring asymptotic flatness), if we suppose that $\Lambda(n)>0$ satisfies an additional smallness constraint that $\Lambda(n)\leq C(n)$, for some constant $C(n)>0$. As a result, under this additional supposition, the special case of Corollary \ref{cfc} in the rotationally symmetric setting would then be a consequence of Theorem \ref{OWt}. Note however that we have not estimated the size of $\Lambda(n)$ in the proof of Corollary \ref{cfc}, and also will not estimate the size of $C(n)$ in the proof of Theorem \ref{C1} below, so whether or not $\Lambda(n)\leq C(n)$ and Corollary \ref{cfc} actually follows from Theorem \ref{OWt} in the rotationally symmetric setting is unclear.

\subsection{Rotationally symmetric AF metrics --- proof of Theorem \ref{C1}}

Below we consider rotationally symmetric metrics on $\mb{R}^n$ in the standard form
\begin{align}
g=dr^2+f(r)^2\ d\sigma^2,
\end{align}
where $r$ is the metric distance to the origin and $d\sigma^2$ is the standard metric on $S^{n-1}$. We assume $f(r)$ extends to a smooth odd function of $r$ with $f'(0)=1$, so that $g_0$ is well defined at $r=0$. When $f(r)=r$ then $g_0$ is exactly the flat metric on $\mb{R}^n$. First we recall the following facts which may be found for instance in \cite{CLN,Ivey}.

\begin{lemma}\label{rotform}
The sectional curvatures of $g=dr^2+f(r)^2\ d\sigma^2$ for the planes tangent to the distance spheres, $\nu_1(r)$, and for the planes containing a radial direction, $\nu_2(r)$, are respectively given by
\begin{align}
\nu_1=\frac{1-(f')^2}{f^2},\quad\nu_2=\frac{-f''}{f},
\end{align}
and the eigenvalues of the curvature operator are $\nu_1$ and $\nu_2$ with multiplicities $\binom{n-1}{2}=\frac{(n-1)(n-2)}{2}$ and $n-1$, respectively.
\end{lemma}

Therefore $\int|\Rm|^{\frac{n}{2}}\ dV_g$ is bounded from below by both $\int|\nu_1|^{\frac{n}{2}}\ dV_g$ and $\int|\nu_2|^{\frac{n}{2}}\ dV_g$. The space $(\mb{R}^n,g)$ possesses a minimal hypersphere exactly when there exists an $R>0$ such that $f'(R)=0$. We will show that this implies $\int|\Rm|^{\frac{n}{2}}\ dV_g$ cannot be too small.

\begin{proof}[Proof of Theorem \ref{C1}]
By our discussion above, it suffices to show that if there exists a function $f:[0,R]\rightarrow\mb{R}$ such that $f(0)=0$ and $f'(0)=1, f'(R)=0$ with $f'(r)>0$ on $[0,R)$, then there exists a $C=C(n)>0$ such that
\begin{align}
\max\left(\int_0^R\left|\frac{1-(f')^2}{f^2}\right|^{\frac{n}{2}} f^{n-1}\ dr,\int_0^R\left|\frac{-f''}{f}\right|^{\frac{n}{2}}f^{n-1}\ dr\right)\geq C(n).
\end{align}
We may rescale by defining $g(r)=\frac{1}{R}f(Rr)$ on $[0,1]$ so that $g(0)=0$ and $g'(0)=1, g'(1)=0$ with $g'(r)>0$ on $[0,1)$, and since $\int|\Rm|^{\frac{n}{2}}\ dV_g$ is scale-invariant, we find that
\begin{align}
E_1(g)&:=\int_0^1\left|\frac{1-(g')^2}{g^2}\right|^{\frac{n}{2}} g^{n-1}\ dr=\int_0^R\left|\frac{1-(f')^2}{f^2}\right|^{\frac{n}{2}} f^{n-1}\ dr,
\\
E_2(g)&:=\int_0^1\left|\frac{-g''}{g}\right|^{\frac{n}{2}}g^{n-1}\ dr=\int_0^R\left|\frac{-f''}{f}\right|^{\frac{n}{2}}f^{n-1}\ dr.\notag
\end{align}
So it suffices to show that the maximum of $E_1(g),E_2(g)$  is bounded below by some $C(n)$. To do this, first we will show that there exists a $C_1(n)>0$ such that if $g(1)=D\geq C_1(n)$, then $E_1(g)\geq 1$. We note $g(r)$ is monotonically increasing on $[0,1]$ and that $D=\int_0^1g'\ dr\leq \left(\int_0^1(g')^2\ dr\right)^{\frac{1}{2}}$, so that
\begin{align}
E_1(g)&=\int_0^1\frac{\left|1-(g')^2\right|^{\frac{n}{2}}}{g}\ dr
\\
&\geq\frac{1}{D}\int_0^1\left|1-(g')^2\right|^{\frac{n}{2}}\ dr\notag
\\
&\geq\frac{1}{D}\left|\int_0^1 (g')^2\ dr-1\right|^{\frac{n}{2}}\notag
\\
&\geq \frac{1}{D}\left(D^2-1\right)^{\frac{n}{2}},\quad\text{if }D\geq 1.\notag
\end{align}
Thus if $D$ is sufficiently large, $E_1(g)\geq 1$ and then $\max(E_1(g),E_2(g))\geq 1$.

Next we estimate $E_1(g), E_2(g)$ from below. Because of our preceding estimates we only need to consider those $g$ for which $g(1)\leq C_1(n)$. For $E_1(g)$ we have similar to before,
\begin{align}
E_1(g)&\geq\frac{1}{C_1(n)}\left|\int_0^1 1-(g')^2\ dr\right|^{\frac{n}{2}},
\intertext{while for $E_2(g)$ we have}
E_2(g)&=\int_0^1\left|\frac{-g''}{g}\right|^{\frac{n}{2}}g^{n-1}\ dr
\\
&\geq\left|\int_0^1 g'' g^{\frac{n-2}{n}}\ dr\right|^{\frac{n}{2}}\notag
\\
&=\left|\int_0^1 \frac{n-2}{n}(g')^2 g^{\frac{-2}{n}}\ dr+\left[g' g^{\frac{n-2}{n}}\right]_0^1\right|^{\frac{n}{2}}\notag
\\
&\geq \left(\frac{n-2}{n}\right)^{\frac{n}{2}}\frac{1}{C_1(n)}\left(\int_0^1(g')^2\ dr\right)^{\frac{n}{2}}.\notag
\end{align}
Hence if $g(1)\leq C_1(n)$ then $\max(E_1(g),E_2(g))$ is also bounded below by some constant depending only on $n$, completing the proof.
\end{proof}

We have seen that $\int|\Rm|^{\frac{n}{2}}\ dV_g$ small alone suffices to guarantee the long-time existence and convergence of this flow to flat space, without a bound on the Sobolev constant as required in Theorem \ref{A}. This is implied either by Corollary \ref{cfc} alone, or by a combination of Theorem \ref{C1} with Theorem \ref{OWt} from \cite{OW}. It would be interesting to see how this situation relates to the general case. Another interesting question is whether $\int|\Rm|^{\frac{n}{2}}\ dV_g$ sufficiently small has any connection to the existence of closed minimal hypersurfaces on a general asymptotically flat manifold.
\end{appendices}

\bibliographystyle{alpha}
\bibliography{references}

\end{document}